\documentclass{siamonline171218}
\usepackage[utf8]{inputenc} 
\usepackage{tikz}
\usetikzlibrary{matrix}
\usepackage{placeins}

\usepackage{graphicx} 

\usepackage[utf8]{inputenc}
\usepackage{subcaption} 

\usepackage{amsmath}
\usepackage{amssymb}
\usepackage{url}
\usepackage{dsfont}

\usepackage{enumitem}
\setlist[enumerate]{leftmargin=.5in}
\setlist[itemize]{leftmargin=.5in}

\newcommand{\vectornorm}[1]{\left|\left|#1\right|\right|}

\newcommand{\R}{\mathds{R}}
\newcommand{\Ms}{\mathcal{M}}
\newcommand{\Ls}{\mathcal{L}}
\newcommand{\Os}{\mathcal{O}}
\DeclareMathOperator{\diag}{diag}
\DeclareMathOperator{\Div}{div}

\numberwithin{theorem}{section}

\title{A Geometrical Method for Low-Dimensional Representations of Simulations\thanks{
\funding{The authors were supported within the projects SIMDATA-NL and VAVID funded by the BMBF, the German Federal Ministry of Education and Research. }}}
\author{Rodrigo Iza-Teran%
\thanks{Fraunhofer Center for Machine Learning and Fraunhofer SCAI, Sankt Augustin, Germany}%
\and
Jochen Garcke%
\footnotemark[2]
\thanks{
	Institut für Numerische Simulation, Universität Bonn, Germany}
}

\headers{A Geometrical Method for Low-Dimensional Representations of Simulations}
{Rodrigo Iza-Teran and Jochen Garcke}

\newtheorem{conjecture}[theorem]{Conjecture}

\usepackage[vlined,ruled,linesnumbered,algo2e]{algorithm2e}
\SetKwInOut{Parameter}{Parameter}
\newcount\Comments  
\newcommand{\kibitz}[2]{\ifnum\Comments=1\textcolor{#1}{#2}\fi}

\begin{document}
\maketitle
\begin{abstract}
We propose a new data analysis approach for the efficient post-processing of bundles of finite element data from numerical simulations. 
The approach is based on the mathematical principles of symmetry. 

We consider the case where simulations of an industrial product are contained in the space of surface meshes embedded in $\R^3$. 
Furthermore, we assume that distance preserving transformations exist, albeit unknown, which map simulation to simulation. 
In this setting, a discrete Laplace-Beltrami operator can be constructed on the mesh, which is invariant to isometric transformations and therefore valid for all simulations. 
The eigenfunctions of such an operator are used as a common basis for all (isometric) simulations. 
One can use the projection coefficients instead of the full simulations for further analysis. 
To extend the idea of invariance, we employ a discrete Fokker-Planck operator, that in the continuous limit converges to an operator invariant to a nonlinear transformation, and use its eigendecomposition accordingly. 

The data analysis approach is applied to time-dependent datasets from numerical car crash simulations. 
One observes that only a few spectral coefficients are necessary to describe the data variability, and low dimensional structures are obtained. 
The eigenvectors are seen to recover different independent variation modes such as translation, rotation, or global and local deformations. 
An effective analysis of the data from bundles of numerical simulations is made possible, in particular an analysis for many simulations in time.
\end{abstract}

\section{Introduction}

In computer-aided engineering (CAE), one models the physical behavior of industrial products by partial differential equations (PDEs), which are then solved numerically.
Nowadays, the art of simulation is highly developed in several industries, where high performance computer systems are used to solve finite element models for many design variants. 
Therefore, engineers need to analyze and compare a large number of simulations in the product development process. Their goal is to get the best product performance, while taking functional constraints, costs, or regulations into account.

To achieve this, a large amount of engineering know-how and time is invested for the evaluation of model variants based on corresponding numerical simulations.
It involves detailed investigation of time-dependent 3D visualizations of the simulated product, e.g. in the form of movements or deformations. 
But analyzing and trying to identify cause and effect using the full finite element mesh---in the range of millions of nodes---is a challenging task.
Due to the difficulty of comparing very fine discretization meshes, usually only a few derived scalar quantities, such as energy absorption or deformations in selected points of the structure, or performance curves, for example vibration response curves, are used for the evaluation of the simulation results. 
This is under the assumption that the relevant information from all simulations can thereby be concentrated.
Furthermore, the input variables together with scalar output quantities allow the construction of response surfaces~\cite{Myers}. 
These approaches fail, if the overall behavior cannot be represented by such scalar quantities or curves. 
In this case, they cannot be used to study complex simulation results in more detail. 

More generally, the overall objective is to evaluate the sensitivity of the simulation results in relation to the input parameters. 
In other domains a range of research on such sensitivity analysis with functional inputs and outputs exists. For example, \cite{Marrel2016} gives an overview of approaches
for spatiotemporal data, with examples for flood risk assessment and radionuclide atmospheric dispersion, and states that further research in this field is strongly needed.
In our targeted engineering domain, only limited research exists which considers the complete numerical results. One of the first works analyzing engineering data from full numerical simulations was \cite{Ackermann}, in which the authors used principal component analysis (PCA) for detecting important parameters from nonlinear finite element simulations. Car crash simulations were analyzed in \cite{Mei2008} using clustering with a local distance measure, where regions of si\-mi\-lar deformation behavior were identified and evaluated. PCA has also been used in \cite{Thole} for a group of simulations where some parameters have been changed. Their approach uses the first eigenvectors of the covariance matrix of all given simulations as deformation modes. 

While linear methods such as PCA have proven to be successful for industrial applications, it is nevertheless known that if nonlinear correlations are present in the data, the use of PCA is not optimal~\cite{LeeV}. 
In \cite{a2013}, several nonlinear methods from machine learning have been used for the analysis of crash simulations. 
Good reconstruction capabilities using a nonlinear principal manifold approach were shown, as well as the detection of the principal effects and their dependencies on input variables by using diffusion maps. Additionally, \cite{Iza2013} shows the application of diffusion maps to engineering data from vibration analysis and metal forming. 
The resulting clustering of vibration response curves or deformation data from simulation bundles show that the most important input variable changes can be identified in low dimensional embeddings.

A related domain is 3D geometry analysis, where the studied object is a surface mesh embedded in $\R^3$. A number of methods have been developed for pose independent shape classification \cite{Rongjie2010}, shape retrieval \cite{Reuter2006}, shape segmentation \cite{reuter2009}, invariant mesh representation \cite{ Lipman}, and compression \cite{Karni2000,Benchen}, to name a few. 
Many of these approaches make use of the Laplace-Beltrami operator, whose eigenvectors are used for pose independent shape recognition and shape retrieval. 
Note that this is exactly the opposite of what we aim for, we would like to distinguish between two isometrically deformed shapes.
In \cite{Ovsjanikov2012}, an approach related to ours, 3D shapes are compared assuming the existence of an a-priori known transformation bringing one shape into another. 
That approach is substantially different from ours, since it describes a shape using features which are then used for classification or pose independent shape matching. 

In this work, we introduce a novel analysis approach, which can efficiently compare finite element simulations given as surface mesh data in 3D. The motivation and idea for our method are inspired by the notion of symmetry. 
Assuming simulations are obtained by specific transformations, we construct operators that are invariant to these transformations. 
From this operator, an orthogonal basis is calculated using its eigendecomposition. We project all simulations onto this basis. 
Thereby, a new representation of the transformed geometries is obtained, with as many coefficients as number of nodes. 
But, only a few spectral coefficients have large values. 
Concentrating on these coefficients, they can be associated with independent effects, for example rotations, translations or types of deformation.
Furthermore, by using different distance measures, which are assumed to be preserved by a transformation, other operators and corresponding eigenfunctions can be obtained. 
We describe in particular a discrete Laplace-Beltrami operator and a discrete Fokker-Planck operator. 
In summary, the high dimensional simulation data can be transformed into a compact representation by using eigenvectors of such operators, enabling an efficient dimensionality reduction, while simplifying further data analysis. 

The underlying ideas for using a spectral decomposition of an invariant operator for data analysis are presented in section~\ref{sec:principles}. 
In section~\ref{sec:operators}, we consider as examples the Laplace-Beltrami operator and a specific Fokker-Planck operator, as well as their discrete versions. 
In particular, we provide first theoretical justifications to the observed decay of the spectral coefficients as well as the obtained separation of independent components. 
An empirical investigation of the method using industrial engineering data from car crash simulations is given in section~\ref{sec:applications}.

\section{Principles of the Approach}
\label{sec:principles}
Our approach is inspired by the notion of symmetry as used for the analytical solution of differential equations.
Here, we have simulations as numerical solutions of partial differential equations.
When a transformation is applied to one simulation, the result is indistinguishable from the original one in the underlying mathematical space.
In other words, it is invariant to the transformation and one could instead work directly in a suitable invariant space.
This notion of symmetry can be exploited in the analysis of simulation results.
However, finding such a transformation for a specific realistic situation as modeled by a PDE is clearly infeasible. 

As an example, we consider simulations of car crashes and the arising deformations of the car parts, where a part is modeled as a (thick) surface embedded in $\R^3$.
Since the underlying transformation is unknown, we propose to employ a differential operator which observes suitable
invariance properties. 
In many practical cases, the variations between different simulations, e.g. due to changes in the material parameters, are distance preserving. 
Here, even after a deformation, the distance between points inside the geometric structure stays the same. 
The Laplace-Beltrami operator is such an invariant operator. It is the same for two surfaces, if one is obtained from the other by an isometric transformation.
Since this operator is positive semidefinite, its spectral orthogonal decomposition can be obtained and an orthogonal basis of an invariant space can be constructed. 
One can now project mesh functions, in particular those from the simulation results, onto the new basis by employing the spectral coefficients.

As we project mesh functions representing isometric deformations onto the invariant basis, a natural decomposition along sets of equivalence classes is obtained. 
In the end, one just has to consider the spectral coefficients along the different components of the orthogonal decomposition. 
They represent so-called orbits, which trace the space of all isometric deformations.
Note that the deformations in a car crash are a complex mixture of different effects such as translations, rotations, different bending behaviors, or torsion. 
How these interact, in particular in relation to the changes in the model design and material parameters, is of great importance for the engineering analysis. 

As an example consider the case of translations and rotations of an object, which is discretised as a 3D mesh. 
The $x$, $y$ and $z$  coordinates of the object are changed, but the distances along the surface are preserved. 
A Laplace-Beltrami operator can be constructed using the 3D coordinates of one of the objects, which will be---approximately in the discrete case---the same for all other objects.
The eigenvectors of this operator can be used as an orthogonal basis. 
Projecting mesh functions $f_x, f_y,$ and $f_z$, representing the $x$, $y$, and $z$ coordinates of the object, onto this basis, one observes that there are two eigenvectors where the projection coefficients vary the most. 
One component represents the offset, while the second component represents the rotation, that is, one obtains independent subspaces for rotations and translations.
To illustrate this, we can take the coefficients from the component that corresponds to rotations of the object. 
Projecting the three mesh functions $f_x, f_y,$ and $f_z$ onto this eigenvector, we obtain three sets of coefficients. 
Plotting these in a 3D plot results in a sampling of the sphere, see figure~\ref{fig:principles_of_approach}.
\begin{figure}
	\centering
		\includegraphics[width=0.3\columnwidth]{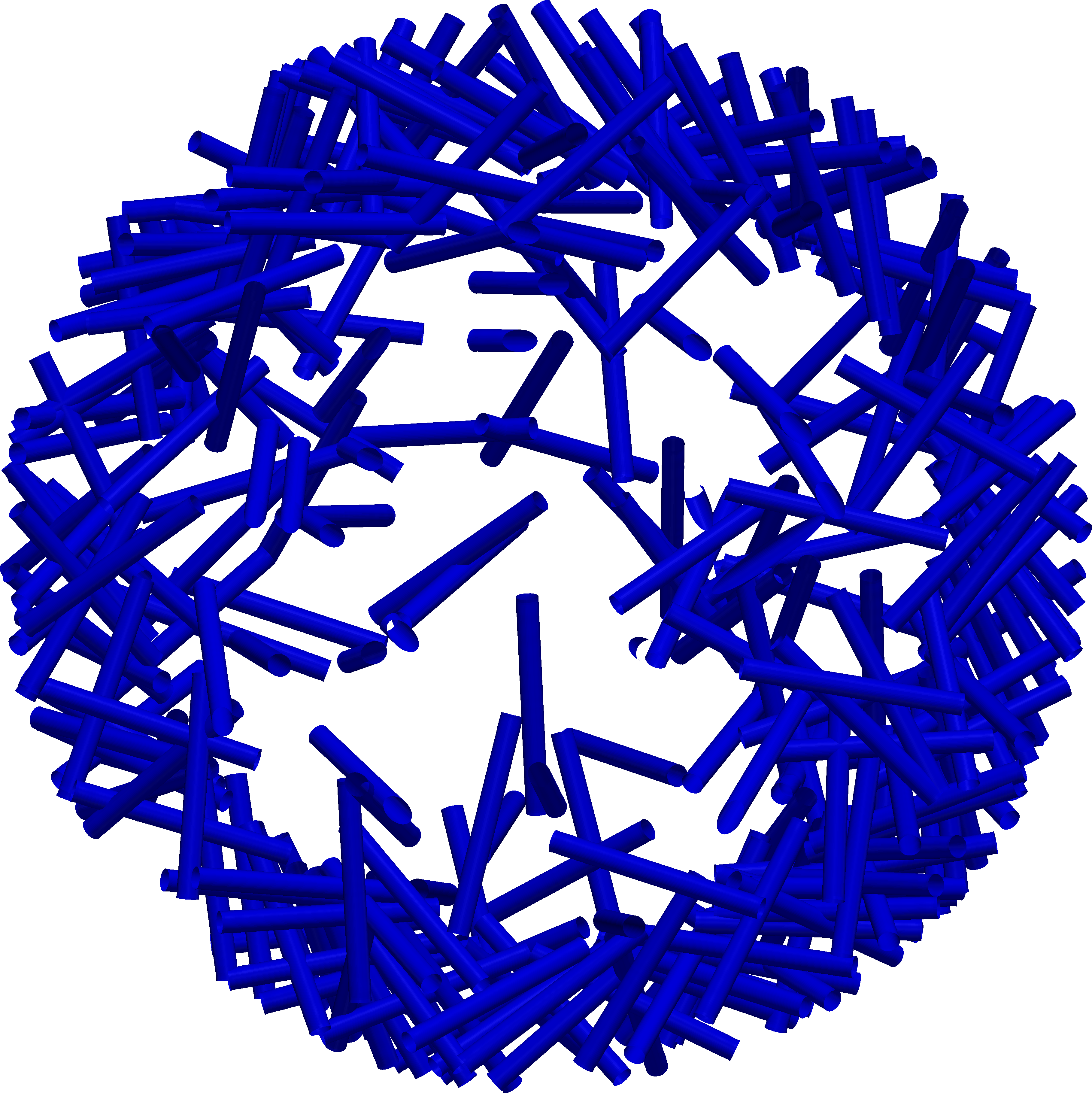} 
\hspace{1cm}
		\includegraphics[height=0.3\columnwidth,width=0.3\columnwidth]{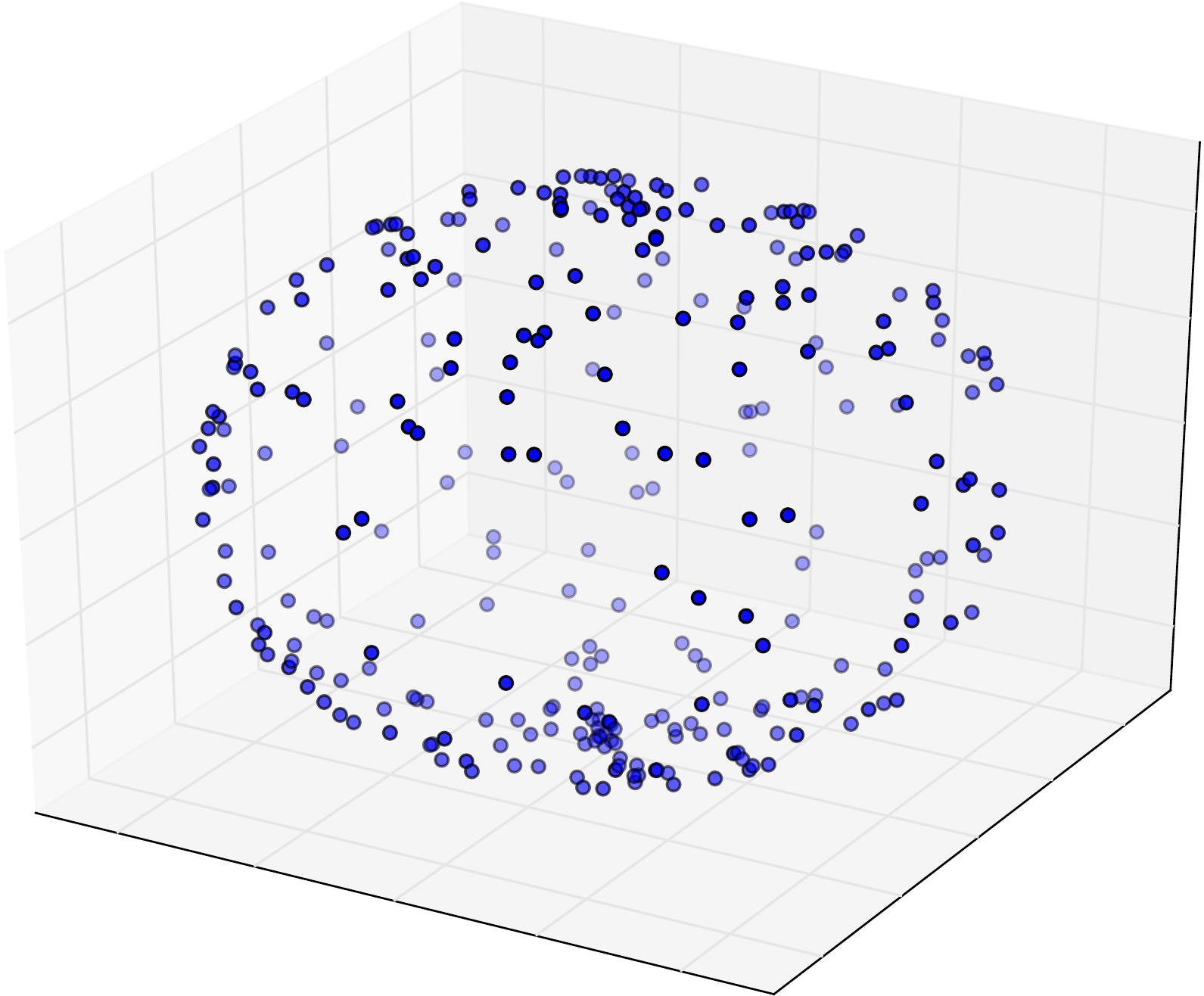}
	\caption{A cylinder is randomly rotated around its axis and translated---for illustrative purposes---with respect to a central point (left).
	The mesh functions for the $x$, $y$, and $z$ coordinates of each rotated and translated cylinder are projected onto the second eigenvector of the Laplace-Beltrami operator. Plotting the resulting three spectral coefficients results in a low dimensional representation of the rotations of the transformed cylinder (right).
	\label{fig:principles_of_approach}}
\end{figure}

For realistic simulation data, we will observe empirically in section~\ref{sec:applications} that the magnitude and the variance of the coefficients decay very fast. 
Furthermore, we can identify specific physical effects, which are represented by components of the spectral decomposition, e.g. independent isometric variations such as translations, rotations, or deformations.
This process can be  thought of as a decomposition into independent 
``physical'' effects. 

We now formalize these considerations.

\subsection{Solution Space of Simulations}
We want to study the space of simulations, that is, the space of time-dependent (numerical) solutions of a partial differential equation on a given fixed domain $\Omega$, where variations are obtained by changing boundary and/or initial conditions. 
In this work, we assume these are modeled as surfaces embedded in $\R^3$. 

First, let us recall definitions from differential geometry. For details refer to, e.g., \cite{Michor}. 
A diffeomorphism is a transformation which is invertible and smooth. 
Let $\varphi: \Ms \rightarrow \Ms'$ be a diffeomorphism between two Riemannian manifolds $(\Ms,g)$ and $(\Ms',g')$, with $g, g'$ the respective metric. 
$\varphi$ is an isometry if it is (geodesic) distance preserving:  $d_{g'}(\varphi(p), \varphi(q)) = d_g(p,q)$.

To formulate the underlying abstract setting we make the following two assumptions:
\begin{itemize}
	\item Each solution at a given time $t$ is a (Riemannian) 2-manifold $\Ms^i(t)$ embedded in $\R^3$.
	\item Solutions can be related by group transformations, which are assumed to be isometries. This means that at each time step $t$, a solution surface $\Ms^i(t)$ is isometric to an initial, or reference, manifold $\Ms$ via a transformation $\Phi^i(t, \cdot)$, i.e.
$
\Ms^i(t) = \Phi^i(t, \Ms).
$
\end{itemize}
We view a numerical simulation result as a surface mesh in time, which is a discrete approximation of a 2-manifold $\Ms^i(t)$. 
For ease of presentation, we drop the dependence on $t$ in the notation from now on. We denote by $\varphi^i: \Ms \rightarrow \Ms^i$ the point wise mapping, which is a bijective isometry under these assumptions.

Now, consider the space of all embeddings of a surface $\Ms$ in $\R^3$, denoted by $Emb(\Ms,\R^3)$.
Ideally, one could work in a smaller space inspired by the framework of shape spaces as given in \cite{Michor2006,Bauer,Bauer2014},
where the shape space is a quotient space of a manifold modulo a specific group action $G$ over it. 
For simulations we would use the space of orbits
\[
Emb(\Ms,\R^3)/G(\Ms),
\]
where $G$ is a group of transformations that leave a surface invariant with respect to distances as measured in the reference surface $\Ms$.
Which particular structure a specific simulation space has will depend on the transformation group $G$. Rotations, translations, and deformations are examples of transformation groups that arise in our application.

\subsection{Discrete Orthogonal Decomposition}
To obtain a new representation for numerical simulation data, we now consider a positive semidefinite operator. 
Therefore, all eigenvalues $\lambda_j$, $ j  \geq 0 $ are real, positive. We furthermore assume they are isolated with finite multiplicity.
For example, the Laplace-Beltrami operator $\Delta_\Ms := - \Div_g \cdot \nabla_g $, from the eigenvalue problem $\Delta_\Ms \psi = \lambda \psi$ restricted to the manifold $(\Ms,g)$, fulfills this assumption~\cite{rosenberg_1997}.
Note the relation between the Laplace operator and Fourier analysis, which is an underlying theme in the following.

The set of normalized eigenfunctions $\left\{ \psi_j \right\}_{j=1}^\infty$ of such an operator forms an orthonormal basis for functions on $\Ms$.
Therefore, for any function $f \in C^{2}(\Ms)$, by projecting the function along the infinite dimensional eigenspace spanned by the eigenfunctions, we can obtain the following decomposition:
\begin{equation} \label{decompos}
f = \sum_{j=0}^{\infty} \alpha_j \psi_j \mbox{, }\quad \alpha_j = \langle f, \psi_j \rangle_g.
\end{equation}
Its vector of spectral coefficients $\alpha = [ \alpha_1, \alpha_2, ...]$, can be used to compute the distance of functions due to the Parseval identity as follows:
\begin{proposition}
	\label{Parseval}
	The difference between two functions $f^1, f^2 \in C^{2}(\Ms) $, 
	$f^i = \sum_{j=0}^{\infty} \alpha_j^i \psi_j$, $\alpha_j^i = \langle f^i, \psi_j \rangle_g $, $\alpha^i = [\alpha^i_1, \alpha^i_2, \ldots ]$, $i=1,2$,
	represented by a decomposition with the normalized eigenfunctions $\left\{ \psi_j \right\}_{j=1}^\infty$ for a positive semidefinite operator, is given by  
	\[
	\| f^1 - f^2 \|_g^2 = \sum_{i,j=0}^{\infty} \langle (\alpha_i^1 -  \alpha_i^2)  \psi_i, (\alpha_j^1 -  \alpha_j^2) \psi_j\rangle_g^2 =  \| \alpha^1 - \alpha^2 \|^2. 
	\]
\end{proposition}

\subsubsection{Discrete Setting}
We view the discretization of a solution of a PDE as a mesh $K$ approximating a $2$-manifold $\Ms$, which is isometrically embedded in $\R^3$. The following definition is one way to quantify such an approximation. See \cite{Belkin2008} for details.
\begin{definition}
Let $K$ be a meshed surface approximating $\Ms$, where the vertices of $K$ are on $\Ms$. 
We say that $K$ is an $(\epsilon, \eta)$-approximation of $\Ms$ if the following two conditions are fulfilled: 
\begin{itemize}
\item For a face $t$ in $K$, the maximum distance between any two points on $t$ is at most $\epsilon \rho$. Here $\rho$ is the \emph{reach}, defined as the infimum of the local distance between any point $w$ in $\Ms$ and the medial axis of the surface $\Ms$.  
\item For a face $t$ in $K$ and a vertex $p \in t$, the angle between $n_t$, the unit outward normal of the plane passing through $t$, and $n_p$, the unit outward normal of $\Ms$ at $p$, is at most $\eta$.
\end{itemize}
\end{definition}
The conditions above state that one needs to have a fine enough mesh to approximate the shape of the surface as well as the curvature.

Furthermore, we define restrictions of continuous functions on $\Ms$ to the mesh $K$:
\begin{definition}
Let $f:\Ms  \rightarrow \R$ be a continuous function on $\Ms$. 
The function $f$ evaluated at the nodes of a mesh $K$ is called a mesh function $f|_K:K \rightarrow \R$.
\end{definition}
\noindent In the following $f$ will denote both the continuous function on $\Ms$ and its restriction to $K$. 

For a discrete operator on a mesh $K$, the eigenvectors ${\psi_j}$ can be calculated as approximations of the eigenfunctions of the corresponding analytical operator. 
Note that we will later in particular consider discrete approximations of the Laplace-Beltrami operator and of a Fokker-Planck operator.
As before, we can project any mesh function $f$ along the eigenvectors and obtain a set of spectral coefficients $\alpha^i_j, j=1,\ldots,N_h$,
where $N_h$ corresponds to the number of nodes on the mesh $K$, which can be very large. 

\begin{corollary}
	\label{proposition_main1}
	Let an orthogonal basis $\{\psi_j\}_{j=1}^{N_h}$ consisting of $N_h$ eigenvectors $\psi_j$ 
	be given, which is obtained from the spectral decomposition of a discrete approximation of a given po\-si\-ti\-ve semidefinite operator on the mesh $K$.
	Then all mesh functions $f$ can be represented as 
	\begin{equation} 
	\label{eq:decomp_prop_main}
	f = \sum_{j=1}^{N_h} \alpha_j \psi_j \mbox{, }\quad \alpha_j = \langle f,{\psi_j}\rangle,
	\end{equation}
	where $ \alpha_j $ are the spectral coefficients obtained by the projection of $f$ into the eigenvector basis.
\end{corollary}
In our application setting, the functions of interest $f: K \rightarrow \R^3$ describe the $x,y,$ and $z$ components of a mesh point at a time step of a numerical solutions of a partial differential equation using the finite element method. Each component function $f_{\{x,y,z\}}$ is represented by (\ref{eq:decomp_prop_main}). Additionally, one is also interested in scalar functions on the surface, such as stress or speed. For a simulation bundle these discrete mesh functions are the high-dimensional data, of dimension $N_h$, which we would like to analyze, e.g. to find clusters or classify results.

\subsubsection{An Operator under Isometric Transformations}
Now, we bring the simulation space and the spectral decomposition of a positive semidefinite operator together by con\-si\-de\-ring isometric transformations.
Let $\varphi^i$ be a distance preserving diffeomorphism $\varphi^i: \Ms \rightarrow \Ms^{i}$ between two Riemannian manifolds $(\Ms,g)$ and $(\Ms^{i},g^i)$. 
Let $\bar{K} =  \left\{ K^i \right\}_{i=1}^{m}$ be a set of meshes which are assumed to have the same connectivity and approximate a set of $\varphi^i$-transformed surfaces $\Ms^{i}$, $i=1,\ldots, m$.
Then $\Ms, \Ms^i$, and the $\Ms^i$ pairwise, are isometric. Assume that $\varphi^i_{|K^i}$, the restriction of $\varphi^i$ to the discrete mesh, where vertices are mapped to vertices, preserves distances as well. In this ideal case the following diagram commutes:
\begin{center}
	\begin{tikzpicture}
	\matrix (m) [matrix of math nodes,row sep=3em,column sep=4em,minimum width=2em]
	{
		\Ms & \Ms^i \\
		K & K^i \\};
	\path[-stealth]
	(m-1-1) edge node [left] {} (m-2-1)
	(m-1-1.2) edge[<->] node [below] {$\varphi^i$} (m-1-2.182)
	(m-2-1.east|-m-2-2) edge[<->] node [below] {}
	node [above] {$\varphi_{| K}^i$} (m-2-2)
	(m-1-2) edge node [right] {} (m-2-2);
	\end{tikzpicture}
\end{center}
As said, the underlying transformations are unknown. Therefore we employ an operator which observes suitable invariance properties. 
It has the same spectral decomposition for the different manifolds $\Ms^i$ and results in an invariant basis for the mesh functions. 
Notice that in a discrete setting it is a-priori not clear how to handle geodesic distance preservation. 
In general, an error is made by their approximate computation and one assumes that it is kept small, which we examine later in section \ref{sec:approx_geo}.

\subsection{Data Analysis Method}
\label{sec:method}
Based on the explained principles we propose a data analysis method for data from numerical simulations, which allows a dimensionality reduction and a separation of effects along components. First results for a theoretical justification of the approach are provided in section \ref{sec:operators}. Numerical observations for realistic cases in automotive product design analyzing data from crash simulations are provided in section \ref{sec:applications}.

\subsubsection{Method Fundamentals}
Given $m$ simulations on a mesh with $N_h$ nodes, we first compute a spectral decomposition of a discrete approximation of an invariant operator. 
In section \ref{sec:operators} of this work, we present two alternatives, the Laplace-Beltrami operator and a Fokker-Planck operator, and their discretizations.
 
In a second step we select $p$ eigenvectors for further analysis. 
For example, we consider the decay of the spectral coefficients of the operator or look at the variance of the spectral coefficients $\alpha^i$ of mesh functions $f^i$ of particular interest. 
Under certain conditions, the coefficients have a strong decay, and therefore the ``essential'' part of a dataset is concentrated on a few coefficients. 
As will be seen in the applications, the first few components can correspond to the dominant behavior, and later ones become less relevant. 
At least for the Laplace-Beltrami operator, the first components can be geometrically associated with high frequency parts or details~\cite{reuter2009}. 

As a consequence, we can analyze the main variability in the simulation data using only the first few spectral coordinates.
Up to dimension three, the projection coefficients can be used directly as embedding coordinates for visualization. 
If more coefficients are required to capture the data, a further method of dimensionality reduction can be used to obtain an embedding~\cite{LeeV}. 
The use of a single spectral basis for all simulations also naturally enables the construction of low dimensional representations for time-dependent problems. 
Thereby a projection of mesh functions for all time steps onto this eigenbasis is possible.

This data analysis procedure has the following properties:
\begin{itemize}
	\item an extraction of useful low dimensional information in the spectral coefficients of largest magnitude and variance,
	\item a natural simultaneous treatment of many time steps,
	\item a separation of independent effects in simulation data, 
    \item a parameterization of all simulation data through the use of the orthogonal basis, which allows a reconstruction in the original high dimensional simulation space,
	\item as a consequence, if the input parameters are correlated with the obtained low dimensional representation, simulations for new input parameter combinations can be approximated, within the originally employed ranges. 
\end{itemize}
Some of these properties will be demonstrated in section \ref{sec:applications} on realistic data.

Under the assumption that we are able to express all solutions of our differential equations as the action of some transformation from a reference configuration, we can think of the proposed method as one that approximates this group action, with the flexibility of being able to define the invariance property as needed.  
As mentioned, we are particularly interested in the mesh functions $f_x$, $f_y$, and $f_z$. Alternatively, the difference of the mesh positions to some other time step, which then gives the deformation, or other mesh functions can be of relevance.

Note here, that our approach is related to methods for nonlinear dimensionality reduction~\cite{LeeV}, in particular diffusion maps~\cite{Coifman} and similar ideas used in diffusion wavelets \cite{Coifman.Maggioni:2004}. 
In \cite{Belkin2004} a decomposition of functions over a manifold using eigenfunctions of the Laplace-Beltrami operator was used in a semi-supervised machine learning context.
These methods work on the manifold of data points, which in our case are the mesh functions describing a simulation. 
In contrast, we consider manifolds stemming from numerical simulations as data points, where the manifolds are subject to an unknown transformation. In some sense, the simulation space can be considered as a manifold of manifolds. As described, we carry out the actual data analysis on mesh functions by taking this additional structure into account.

Observe that, in the other methods,
one calculates a suitable distance between all data points to perform the actual dimensionality reduction. 
By the proposed approach, a distance calculation between the data points is now available through the spectral projection coefficients. 
The resulting distance can then replace the Euclidean distance in the other methods for further processing.
Furthermore, our method allows an easy reconstruction in the high-dimensional space using the orthogonal eigenvectors. 
To our knowledge, this important property is only shared by PCA, on which our empirical comparisons therefore focus.

\subsubsection{Comparison with Principal Component Analysis}
PCA is a very efficient way of representing the variability of a dataset $X = \{X^1, X^2, \allowbreak \ldots, X^m \}$, $X^i \in \R^n $. 
It is obtained from an eigenvalue decomposition given by $[U S] = \text{EVD}(X_c X_c^T) $, where $X_c$ 
is a centered representation of $X$, $U$ collects the eigenvectors, and $S$ is a diagonal matrix whose diagonal elements are the eigenvalues.
If the eigenvalues are ordered in descending order, then the first eigenvector corresponds to the first main variation of the dataset, the second to the second one, and so on. Often only a few coefficients are necessary to reconstruct each data so that $X^i \approx    \sum_{k=1}^p  \langle U_k,X^i\rangle U_k$ with $ p < m \ll n$. 
Note that underlying the PCA is the preservation of Euclidean distances in the high dimensional space~\cite{LeeV}. 
The embedding can also be computed by the singular value decomposition of $X_c$, so that the computational effort scales mainly with the minimum of $m$ and $n$.

In the case of numerical simulations in a finite element space, each simulation can be considered as a point in $\R^{n}, n= 3\cdot N_h$.
In view of our application, with PCA, both the basis, that is the principal components, and the projection coefficients depend on this high-dimensional data from the mesh functions.
Comparing the basis from PCA and an invariant operator, e.g. the Laplace-Beltrami operator, one observes significant differences. 
PCA concentrates the data in at most $m$ coefficients, where the variability of the data plays the central role, but neither geometry nor smoothness of the data objects is considered. 
While the first components of the PCA can encode much information, this is under the assumption that there is a linear transformation between high and low-dimensional space.  
If there is some nonlinear dependence on parameters, then other methods than the PCA are needed. 
In this case the data resides on a nonlinear manifold.

\section{Employed Operators}
\label{sec:operators}

Taking a closer look at the proposed method, the important step is the construction of the discrete invariant differential operators, which have a special behavior under isometric transformations.
For example, pose invariance of a deformation can be achieved using an approximation of the Laplace-Beltrami operator constructed using geodesic distances. It has been observed that the employed discrete Laplace-Beltrami operator has a discrete spectrum and that the eigenvectors build a geometric decomposition of ``frequencies''~\cite{reuter2009}. 

Furthermore, other operators that are invariant to specific transformations can be constructed based on suitable distance measures on the mesh.
In particular, we introduce a Fokker-Planck operator, where one can approximately recover a problem-dependent metric, and with it, define an operator that is invariant to an underlying nonlinear transformation. 
Under suitable conditions, a fast decay of the spectral coefficients exists. 
Thereby, also in this case, most of the ``energy'' of the data is concentrated in the first few components.
Other invariant operators can be used as well, obtained by different distance measures for example.

\subsection{The Laplace-Beltrami Operator on a Mesh}
\label{sec:Lap_mesh}
As before, let $\Ms$ be a surface. 
Let $K$ be an $(\epsilon, \eta)$-approximation of $\Ms$. 
Further, let $t$ be a face in $K$, let $\#t$ denote the number of vertices on $t$, and let $V(t)$ be the set of vertices of $t$. Then, for any vertex $w \in V(t)$ of any face $t \in K$, following \cite{Belkin2008} a mesh Laplace operator can be defined as
\begin{equation}
	\label{LB_approx}
	L_{K}^h f(w) = \frac{1}{4 \pi h^2} \sum_{t \in K} \frac{Area(t)}{\#t} \sum_{p \in V(t) } e^{-\frac{d(p , w)^2}{4 h}} (f(p) - f(w)).
\end{equation}
Here, $h$ is a parameter which corresponds to the size of the local neighborhood at a point and $d(p,w)$ denotes the graph distance, that is, the length of the shortest path between $p$ and $w$ along the graph, using Euclidean distances between the vertices as graph weights.

The following theorem exploits the fact that the Euclidean distance is locally a good approximation of the geodesic distance. 
One obtains a result for the approximation of the Laplace-Beltrami operator $\Delta_{\Ms}$ on a surface by the mesh Laplace operator $L_{K}^h$ on a corresponding mesh.
\begin{theorem}
	\label{LB_approx2}
	\emph{(Laplace-Beltrami Approximation \cite{Belkin2008})}
	\label{LaplaceBeltrami}
	Let $K_{\epsilon,\eta}$ be a an $(\epsilon, \eta)$-approx\-imation of $\Ms$. Put $h(\epsilon,\eta)= \epsilon^{\frac{1}{2.5 + \alpha}}+\eta^{\frac{1}{1+\alpha}}$ for an arbitrary positive number $\alpha > 0$. Then for any function $f \in C^2(\Ms)$ it holds that
	\[
	\lim_{\epsilon,\eta \rightarrow 0} \sup_{K_{\epsilon,\eta}}  \vectornorm{ L_{K_{\epsilon,\eta}}^{h(\epsilon,\eta)} f - \Delta_{\Ms} f|_{K_{\epsilon,\eta}} }_{\infty} = 0,
	\]
	where the supremum is taken over all $(\epsilon, \eta)$-approximations of $\Ms$.
\end{theorem}
Theorem \ref{LB_approx2} was derived in \cite{Belkin2008} for unbounded manifolds. However, the result holds for interior points of a surface with boundary. 
Moreover, the Euclidean distance is used in \cite{Belkin2008} in the definition of $L_{K}^h$ instead of $d(p,w)$. 
For an $(\epsilon, \eta)$-approximation of $\Ms$, the graph distance is a better approximation of the geodesic distance than the Euclidean distance. 
Therefore, the theorem naturally holds for definition (\ref{LB_approx}) as well~\cite{DIzaTeran}.
As a consequence, for a mesh fine enough that also approximates the curvature of $\Ms$ well, we expect to get an approximation of the corresponding continuous Laplace-Beltrami operator on this surface. 

These considerations justify that the simulation bundle can be jointly projected onto the eigenbasis obtained from one discrete operator.
We make use of expression (\ref{LB_approx}), evaluated on a reference mesh. 
The graph distance, as an approximation of the geodesic distance, is calculated using the algorithm described in \cite{Mitchell}. To approximate the Laplace-Beltrami operator~(\ref{LB_approx}), we use a slightly modified version of the software\footnote{J. Sun, MeshLP: Approximating Laplace-Beltrami Operator from Meshes, \url{geomtop.org/software/meshlp}}
used in \cite{Belkin2008}. That algorithm requires a parameter $ \rho \cdot \sqrt{h}$ which controls the maximum geodesic distance of nodes which are employed in the calculation of a matrix entry. The overall procedure is described in algorithm~\ref{alg:LaplaceBeltrami}.

\begin{algorithm}[t]
	\DontPrintSemicolon
	\SetKwComment{tcmy}{$\triangleright$ }{}
	\KwIn{simulation $x^r$: a triangular mesh in $\R^3$ with $N_h$ points and $N_f$ faces} 
	\Parameter{$\rho$, $h$, $p$}
	\KwOut{first $p$ eigenvectors of the discrete Laplace-Beltrami operator}
	
	\ForEach(\tcmy*[f]{estimate areas of each face}){$\tilde k \in \{1,\ldots,N_f\}$}{
		\ForEach(\tcmy*[f]{$1/3$ of area assigned to each vertex}){$i \in \{1,2,3\}$}{
			$area[\tilde k_i]$ = (area of face $\tilde k$)/3\tcmy*[r]{face $\tilde k$ vertices indexed by $\tilde k_i$}
		}	
	}
	\ForEach(\tcmy*[f]{weight matrix with graph distances}){$k \in \{1,\ldots,N_h\}$}{
		$[ids,dists]$ = graphdist($k, x^r, \rho \cdot \sqrt{h}$)\tcmy*[r]{distances on $x^r$ up to $\rho \cdot \sqrt{h}$}
		\ForEach{$l,d \in [ids,dists]$}{
			$\mathbf{W}[k,l] = area[k] \cdot area[l] \cdot \exp(-d^2/(4h))  / (4 \pi h^2)$\;
			
		}
	}    
	$\mathbf{L}  =  \mathbf{W} - \mathbf{D} $, where $\mathbf{D} = diag(\mathbf{W} \cdot \mathds{1})$ \tcmy*[r]{compute graph Laplacian}
	decompose $ \mathbf{L}$  by $[\mathbf{U},\mathbf{E}] = \text{EVD}(\mathbf{L}) $\tcmy*[r]{compute eigendecomposition}
	\Return first $p$ non-trivial eigenvectors $\mathbf{U} $\;   
	\caption{Spectral decomposition of the discrete Laplace-Beltrami operator}
	\label{alg:LaplaceBeltrami}
\end{algorithm}

\bigskip

\label{sec:approx_geo}
We now study the effect of approximating the geodesic distance using the graph distance on the mesh.
Here, we assume that a point on a mesh is sampled from the surface. 
The discrete nature of the mesh results in an approximation to the geodesic distance on the surface. 
For an $(\epsilon, \eta)$-approximation of $\Ms$, the difference between geodesic and graph distance can be made arbitrarily small with decreasing $\epsilon$ and $\eta$~\cite{DIzaTeran}.

A bijective transformation $\varphi: \Ms \mapsto \Ms'$ between manifolds $\Ms, \Ms'$, is said to be an $\varepsilon$-isometry if
\begin{equation*}
	\sup_{p, w \in \Ms} | d_{\Ms}(p,w) - d_{\Ms'}(\varphi(p), \varphi(w)) | \leq \varepsilon.
\end{equation*}

We now assume that the discrete transformations $\varphi_{| K}^i$ between (discrete) surfaces are $\varepsilon$-isometric. 
Under the assumption that the variations in the distances due to the approximation have a Gaussian distribution, 
we can state the following result on the effect of errors in the distance computations on the corresponding discrete Laplace-Beltrami operator. 
\begin{proposition}
	\label{propositionLB}
	1) Let a surface $\Ms$ be given and let a geodesic distance preserving mapping $\varphi^i$ transform $\Ms$ into a series of deformed surfaces $\Ms^i$, $i=1,\ldots, m$. The Laplace-Beltrami operator constructed using the geodesic distance in $\Ms$ is the same as the one for each $\Ms^i$.
	
	2) Let the set of meshes $\bar{K}=  \left\{ K^i \right\}_{i=1}^{m}$ contain the approximations of the surfaces $\Ms^i$. Let there be transformations $\varphi_{|K}^i, i=1,\ldots,m$ between the meshes which are $\varepsilon$-isometric, where the $\varepsilon$ will depend on the two respective meshes involved.
	Further, assume that the variance for the $\varepsilon$-isometric transformations follow a Gaussian distribution. 
	Then, the approximation of the Laplace-Beltrami operator $L_{K}^h$ given in  (\ref{LB_approx}), constructed using graph distances for one mesh $K$, differs only by a scaling factor from the ones constructed using one of the deformed meshes $i=1,\ldots,m$ in the set $\bar{K}$.
\end{proposition}
\begin{proof}[Proof sketch]
	The first statement is obvious. Since the geodesic distance stays the same after an isometric transformation, calculating the Laplace-Beltrami operator based on it will lead to the same result.
	
	For the second statement, we are in the discrete case, so the geodesic distance is approximated by the graph distance, where we assume a small error perturbation of magnitude $\epsilon$ which follows a Gaussian distribution. 
	Now, we can use a result from \cite{Karoui2010}, which states that for the Gaussian kernel, as we use here, such a Laplacian matrix disturbed by noise can be considered a rescaled version of the original one.
	For the full proof, see \cite{DIzaTeran}.
\end{proof}

\subsection{Fokker-Planck Operator}
\label{sec:FP_operator}
In addition to the use of the Laplace-Beltrami operator, operators with other types of invariances are possible. 
We introduce a specific Fokker-Planck operator, which is based on ideas developed in the context of nonlinear independent component analysis (NICA)~\cite{Singer_nonlinear,Kushnir2012280}.
The underlying model for the data assumes independent stochastic It\^{o} processes $p^{(i)}$, given by
\[
dp^{(i)} = \alpha^i(p^{(i)}) dt +b^i(p^{(i)}) d w^{i}, \,\, i = 1, \dots N_h, 
\]
where $\alpha^i$ and $b^i$ are unknown drift and noise coefficients, and $\dot{w}^i$ are independent white noises.
That is, the $w^i$ are Wiener processes. 
Given $N_h$ points $p^{(1)}, \cdots , p^{(N_h)} \in \R^M$ in the unobservable space $\mathcal{S}^p$, they are mapped to points $\eta^{(1)}, \cdots , \eta^{(N_h)} \in \R^{d}$ in another space $\mathcal{S}^{\eta}$ by a nonlinear transformation $\varphi$, $\eta^{(k)} = \varphi(p^{(k)})$. One assumes that an observation of the $N_h$-dimensional process $p = (p^{(1)}, p^{(2)}, \dots, p^{(N_h)} )$ is not possible, but that the result of a (nonlinear) mapping $\varphi:\mathcal{S}^p \rightarrow \mathcal{S}^\eta$ is available.

Let us connect this to the analysis of simulations in engineering. As an example we consider again crash simulations for cars. 
In the industrial design work flow, a numerical simulation is performed for a given set of design parameters observing the equations of structural mechanics and computing deformations over time. For a fixed time step, we can interpret the deformed mesh as the result of a mapping $\varphi$ from the initial mesh.
Now, besides the deterministic laws of structural mechanics, there are additional stochastic or uncertain effects.
These might be due to (random) variations of the design parameters, as is the case for reliability design or robustness studies, due to numerical errors or instabilities, due to uncertainties in the modeling assumptions, such as the material models or parameters, or other effects. 
In other words, the numerical result at a given time step is the result of deterministic effects---modeled by the mapping $\varphi$---and stochastic effects, modeled by a stochastic  It\^{o} process $p^{(i)}$ per mesh point.
Surely, whether the assumption of underlying It\^{o} processes, and in particular their independence, is valid for this practical setting is rather speculative. However, note that there are approaches in structural dynamics that use this assumption. For example, see~\cite{caddemi} for the It\^{o} formulation characterizing accumulated structural deformations. 
We conjecture that, at least for small time steps, the variation in the numerical results due to variations in material parameters can be viewed in a stochastic setting, while the independence should be viewed in an approximate fashion. 

We now sketch results from \cite{Singer_nonlinear,Kushnir2012280}, which allow an approximate distance computation in the unobservable space $\mathcal{S}^p$ by using data samples in the observable space $\mathcal{S}^{\eta}$. Considering a Taylor series of the inverse $\varphi^{-1}$ at $(\eta^{(k)}+\eta^{(l)})/2$, one can derive linear approximations to $p^{(k)}$ and $p^{(l)}$ and estimate their Euclidean distance by
\[
\|  p^{(k)} - p^{(l)} \|_{\R^M}^2 = (\eta^{(k)} - \eta^{(l)}  )^T \left[ (J_{\varphi} J_{\varphi}^T)^{-1}\left(\frac{\eta^{(k)} + \eta^{(l)}}{2}\right) \right] ( \eta^{(k)} - \eta^{(l)})  + O(\| \eta^{(k)} - \eta^{(l)}  \|_{\R^d}^4). 
\]
While we do not have access to $J_{\varphi}J_{\varphi}^T$, one can approximate it using the local covariance matrix of the observed variables using
$$
J_{\varphi}(p^{(i)})J_{\varphi}(p^{(i)})^T  = \frac{d+2}{\delta^2} C_{i,\delta} + O(\delta),
$$
where $\delta$ is the size of a ball around $p^{(i)}$. The underlying procedure is no other than a local PCA of the ellipsoid point cloud, for details see \cite{Singer_nonlinear}. 

Finally, since $(J_{\varphi} J_{\varphi}^T)^{-1}$ at $(\eta^{(k)}+\eta^{(l)})/2$ cannot be obtained that way, a second order ap\-pro\-xi\-ma\-tion of it is employed by averaging $(J_{\varphi} J_{\varphi}^T)^{-1}(\eta^{(k)})$ and $(J_{\varphi} J_{\varphi}^T)^{-1}(\eta^{(l)})$. 
We summarize this approximation in the following proposition. See \cite{Singer_nonlinear, Kushnir2012280} for details.
\begin{proposition}
	\label{propositionNICA}
	Let $p,p' \in \mathcal{S}^p \subset \R^{M}$, and let $\eta, \eta'$ be their respective mappings to the observable space $\mathcal{S}^\eta \subset \R^{d}$, $M \leq d$. Then the distance in $\mathcal{S}^p$ (using local coordinates) can be approximated as:
	\begin{equation}
		\label{distNICA}
		d(p,p')^2 := 2 (\eta - \eta'  )^T \left[ J_{\varphi}J_{\varphi}^T(\eta) + J_{\varphi}J_{\varphi}^T(\eta') \right]^{-1} ( \eta - \eta' ),
	\end{equation}
	where $J_{\varphi}$ is the Jacobian of the transformation. It holds that
	\begin{equation}
		\|  p -  p' \|_{\R^M}^2 = d(p,p')^2 + O(\| \eta - \eta'  \|_{\R^d}^4). 
	\end{equation}
	The result also holds for the alternative approximation
	\begin{equation}
		d(p,p')^2 := \frac{1}{2} (\eta - \eta'  )^T \left[ (J_{\varphi} J_{\varphi}^T)^{-1}(\eta) +  (J_{\varphi} J_{\varphi}^T)^{-1}(\eta') \right] ( \eta - \eta' ).
		\label{distNICA_Alt}
	\end{equation}
\end{proposition}
Using the distance approximation (\ref{distNICA}) or (\ref{distNICA_Alt}), we define the density normalized weight matrix $W_{d}$ 
\begin{equation}
	W_{d} := D_A^{-1/2} A D_A^{-1/2}, \quad \text{ where } \quad A_{k,l} = \exp \left(- d\left(p^{(k)}, p^{(l)}\right)^2 / \epsilon\right),
	\label{normalized_nica}
\end{equation}
with $D_A := \diag(A \cdot \mathds{1})$ and $\mathds{1}$ is the vector of all ones.
Further, $W_{d}$ can be transformed into 
\begin{equation}
	W_{rs} := D_d^{-1} W_{d} \quad \text{ and } \quad W_s = D_d^{-1/2} W_{d} D_d^{-1/2},
	\label{row_stocha}
\end{equation}
with $D_d := \diag(W_d \cdot \mathds{1})$. 
The row stochastic matrix $W_{rs}$ is similar to the symmetric matrix $W_s$ in the sense that they
share the same eigenvalues and the corresponding eigenvectors are related by $\psi_s = D_{d}^{1/2} \psi$, with $\psi$ denoting an eigenvector of $W_{rs}$ and $\psi_s$ an eigenvector of $W_s$~\cite{Chung:1997}.

It has been demonstrated in \cite{Singer_nonlinear} that the discrete operator
\begin{equation}
	L =  
	W_{rs} - I,
	\label{rowstochastic_nica}
\end{equation}
for $A$ as in (\ref{normalized_nica}), converges to a Fokker-Planck operator in the non-observable space $\mathcal{S}^p$
\begin{equation}
	\label{Nica_continuous}
	\mathcal{L}_{\mathcal{S}^p} f= \Delta f - \nabla U \cdot \nabla f, \text{ with } \quad U = -2 \log \mu.
\end{equation}
Here $\mu$ is the density of points in $\mathcal{S}^p$, which is (re)constructed from the density of data points in the observable space and the implicitly observed transformation $\varphi$. 
Therefore, through this construction we are approximating an operator in $\mathcal{S}^p$. 
Under suitable conditions it has been demonstrated that the eigenvector corresponding to the first non-trivial eigenvalue of the operator (\ref{rowstochastic_nica}) is actually a function of the first non-observable variable, the second eigenvector is a function of the second non-observable variable, and so on~\cite{Singer_ica}.

As outlined, we propose to use this setting for numerical simulations of a PDE model subject to parameter changes and other stochastic effects. 
To construct the discrete Fokker-Planck operator, we take the nodes of a surface mesh $K$. 
We are considering a simulation bundle where each simulation is assumed to be obtained through a stochastic realization $\varphi^i$ of a nonlinear transformation  $\varphi$ from a reference mesh. 
In this setting, a set of $m$ simulations is available to us, and based on this information, we can now construct an operator which is invariant to the nonlinear transformation $\varphi$. 
One starts with a reference mesh and considers the simulation results at a time step as a cloud around each mesh point. 
Then from this cloud, the local sample covariance matrix $C^{(k)}$ is obtained and one approximates $J_{\varphi}(p^{(k)})J_{\varphi}^T(p^{(k)})$ and its inverse.
This allows us to evaluate expression (\ref{distNICA}) locally at each point $p^{(k)}$ of the reference simulation. 
In some sense, we estimate the deterministic effect over the bundle of numerical simulations by this procedure. 
Most noteworthy, the obtained discrete Fokker-Planck operator can be considered to be invariant to the transformation $\varphi$, since we used a data-driven distance which is invariant to the deterministic effect. Therefore it is the same operator for all of the $m$ surfaces stemming from the numerical simulations.

Algorithm \ref{alg:algNICA} describes the general procedure for the evaluation of this operator. 

\begin{algorithm}
	\DontPrintSemicolon
	\SetKwComment{tcmy}{$\triangleright$ }{}
	\KwIn{
		data set bundle $\{x^j\}, j=1, \ldots, m$ of surfaces embedded in $\R^3$, where each $x^j$ is given by $N_h$ points $p^{j,(k)}$
	} 
	\Parameter{$\epsilon$, $p$}
	\KwOut{first $p$ eigenvectors of a discrete Fokker-Planck operator}
	\ForEach(\tcmy*[f]{estimate local Jacobi matrices at $p^{(k)}$}){$k \in \{1,\ldots,N_h\}$}{
		$\mathbf{JJt}[k] = \mbox{(pseudo)inverse of } C_{k}$, with $C_{k}$ local sample covariance matrix from $p^{j,(k)}$ \;
	}
	\ForEach(\tcmy*[f]{calculate weight matrix}){$l,k \in \{1,\ldots,N_h\}^2$}{
		$\mathbf{A}[k,l] = \exp(-d(p^{(k)},p^{(l)})^2/\epsilon)$, where $d(p^{(k)},p^{(l)})$ after (\ref{distNICA}) and using $\mathbf{JJt}[k]$\;
	}
	compute $W_{rs}$ and $D_{d}$ using (\ref{normalized_nica}), (\ref{row_stocha}) \tcmy*[r]{density normal.\enspace row stoch.\enspace weight matrix}
	solve $[\mathbf{U},\mathbf{E}] = \text{EVD}(W_{rs})$ \tcmy*[r]{compute eigendecomposition}
	\Return $p$ first non-trivial eigenvectors $D_{d}^{1/2} \mathbf{U}$\;   
	\caption{Spectral decomposition of a discrete Fokker-Planck operator}
	\label{alg:algNICA}
	
\end{algorithm}

\bigskip
\label{sec:decay}
As explained, we use a spectral decomposition of an operator that is invariant under a specific transformation, 
and project a set of mesh functions along the resulting basis. 
We now conjecture that, under certain assumptions, one can expect a strong decay of the spectral coefficients depending on the smoothness properties of the data. 
 
\begin{conjecture}[Decay of the Spectral Coefficients]
	Let the assumptions from section \ref{sec:FP_operator} hold and let the orthogonal decomposition based on a Fokker-Planck operator as in corollary \ref{proposition_main1} be given. Then, for smooth functions $f^i: \Ms^i \rightarrow \R, f^i \in C^k, i=1, \ldots, m$, the spectral coefficients $\alpha_j^i = \langle f^i|_K, \psi_j \rangle, j=1, \ldots, N_h$ of their mesh representation $f^i|_K$ decay as
	\[
	| \alpha_j^i |   \leq \frac{C}{(\gamma_j)^{k}} \| (f^i)^{(k)} \|_{L^2}, 
	\]
	depending on the degree of smoothness $k$ of the functions $f^i, i=1,\ldots,m$, and $\gamma_{j+1} \geq \gamma_j$, where the $\gamma_j$ are independent of $f^i$. 
	\label{decaybig}
\end{conjecture}
At least for a simplified setting, this partially explains the observed strong decay of the spectral coefficients. 
For a theoretical justification of this conjecture, we employ the following observation from \cite{Singer_ica}.
\begin{proposition}
	\label{Sturm}
	Let the assumptions from section \ref{sec:FP_operator} hold. Given the eigenvalue problem (\ref{Nica_continuous}) for the Fokker-Planck operator $\mathcal{L}_{\mathcal{S}^p}$, the operator can be separated as
	\begin{equation}
	\mathcal{L}_{\mathcal{S}^p} = \sum_{l=1}^{d} \mathcal{L}_l,
	\label{eq:FP_separated_sum}
	\end{equation}
	where each $\mathcal{L}_l$ is a one-dimensional Fokker-Planck operator in an interval $(a_l, b_l)$ with specific Neumann boundary conditions. 
	The eigenfunctions of $\Ls_{\mathcal{S}^p}$ are tensor products of the one-dimensional eigenfunctions for the $\mathcal{L}_l$. 
\end{proposition}
Note that the operator in expression (\ref{Nica_continuous}) is the one on the unobservable manifold $\mathcal{S}^p$, which is assumed to be a subset of $\R^d$. 
The observable space is nonlinearly transformed, and we will approximate the operator in the observable space. 
The unobservable manifold on which we are considering the Fokker-Planck operator $\Ls$ is planar. 
Accordingly, the data density is a product of $d$ one-dimensional densities, due to the assumed independence of the underlying It\^{o}-processes.
The full derivation can be found in \cite{Singer_ica}. See also \cite{Singer_nonlinear,Kushnir2012280}.

The following result is from the spectral approximation theory for Sturm-Liouville problems. See \cite[Chapter~5.2]{canuto} for details.
\begin{proposition}
	\label{Exponential}
	Let  $f \in C^k$, and its spectral coefficients  $ \hat{f}_j = \langle f, \hat \psi_j \rangle $, where $\hat \psi_j$, $j=\left\lbrace 0,1,\ldots,\right\rbrace$ are the eigenfunctions of the one-dimensional Sturm-Liouville eigen\-value problem 
	\begin{equation}
		\frac{d}{d x} \left( e^{-U}   \frac{d\hat  \psi}{d x} \right) + \lambda  e^{-U} \hat \psi = 0 \; \; x \in (-1,1),
		\label{sturm_1d1}
	\end{equation}		
	with eigenvalue $\lambda_j$.
	Provided $p$ is zero at the boundary, or specific boundary conditions are satisfied, it holds that
	\begin{equation}
		| \hat{f}_j |   \leq \frac{C}{(\lambda_j)^{k}} \| f^{(k)} \|_{L^2 (-1,1)}. \label{eq:sturm_liouvill_bound}
	\end{equation}		
\end{proposition}
Note that one often has $\lambda_j = \Os(j^2)$, so that the denominator would be $j^{2k}$.

\begin{proof}[Justification of conjecture \ref{decaybig}]
	We provide a sketch of a path to prove the conjecture. 
	First, the  decomposition in expression (\ref{eq:decomp_prop_main}) is constructed using approximations of the eigenfunctions of the Fokker-Planck operator $\Ls_{\mathcal{S}_p}$ of independent stochastic It\^o processes. 
	In our case, this operator can be separated into independent one-dimensional Sturm-Liouville problems due to proposition \ref{Sturm}~\cite{Singer_ica}. 
	As also shown in \cite{Singer_ica}, the first eigenfunctions of the operator given in (\ref{Nica_continuous}) provide the independent components and they correspond to the eigenfunctions of the Sturm-Liouville problem. 
	Now, proposition \ref{Exponential} establishes an estimation of the decay rate of spectral coefficients, depending on the type of Sturm-Liouville problem and its boundary conditions. Assuming it can be applied in our setting, the decay of the spectral coefficients $\hat f^i_j, j=1, \ldots N_h$ with regard to the spectral decomposition of the one-dimensional components $\Ls_l$ from (\ref{eq:FP_separated_sum}) can be approximated as in (\ref{eq:sturm_liouvill_bound}). Suitably combining the bounds for the individual $\Ls_l$ one can achieve, where the $\gamma_j$ depend on the $\lambda^l$s,
	\[
	| \hat f^i_j |   \leq \frac{C}{(\gamma_j)^{k}} \| (f^i)^{(k)} \|_{L^2}.
	\]
	
	Clearly, proposition \ref{Exponential} cannot be used here directly since only discrete data is available.
	It is an open question how to bridge the gap between the analytical result and the discrete case. 
	However, there are several recent works on point wise estimates between graph Laplacians and continuum operators or their spectral convergence, e.g.~\cite{Singer2016, Trillos.Slepcev:2016} and the references therein. 
	Currently, however, only consistency results for eigenvectors and eigenprojections have been obtained. 
	Based on these and related works, and connecting them with standard approximation results for functions on surfaces and discrete meshes, we conjecture that bounds on the spectral coefficients of discrete functions as above are attainable.
\end{proof}

\section{Data Analysis for Finite Element Simulation Bundles}
\label{sec:applications}
Finite element simulations in industry are nowadays used to study the behavior of physical objects. 
Consider our running example of car parts under deformations due to a crash. 
The development of new car models demands the creation of many simulations with variations in material parameters or the geometry.
These are analyzed and modified based on engineering judgment in a time consuming research and development process.
The deformations in a car crash are a complex mixture of different effects such as translations, rotations, different bending behaviors, or torsion. 
How these interact, in particular in relation to the changes in the model design and material parameters, is of great importance for the engineering analysis. 
We now study the application of the introduced data analysis method to such crash simulation results.

We use simulation data from a frontal crash simulation of a Chevrolet C2500 pick-up truck, a model with around 60,000 nodes from the National Crash Analysis Center\footnote{\url{http://www.ncac.gwu.edu/vml/models.html}, accessible via \url{http://web.archive.org/}}. 
We performed $m=126$ simulations\footnote{Computed with LS-DYNA \url{http://www.lstc.com/products/ls-dyna}} of a frontal vehicle crash with $\tau=17$ (saved) time steps, where the thicknesses of nine parts are varied randomly by up to $\pm 30\%$~\cite{a2013}. The parts subject to thickness changes are shown in the lower part of figure \ref{fig:truck_crash}.
The variation in the thickness of these nine parts results in different deformations of the original structure.

We mainly investigate the mesh functions $f^i_x, f^i_y,$ and $f^i_z$, one for each coordinate $x, y,$ and $z$, which describe the deformation of the mesh for a given simulation for a selected time step.  

\begin{figure}
	\begin{center}
		\includegraphics[width=0.8\columnwidth]{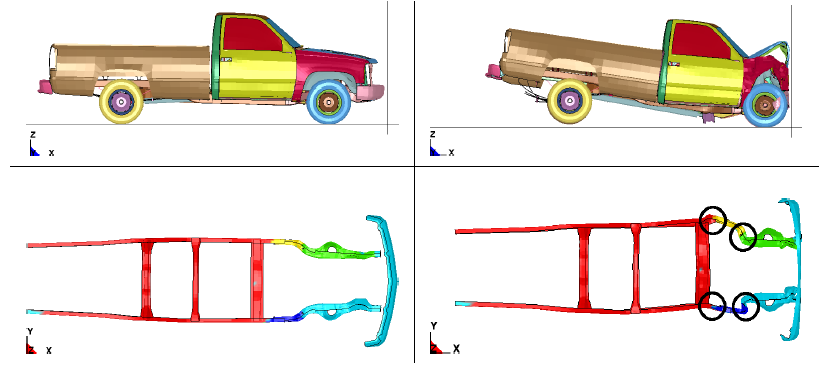}
	\end{center}
	\caption{Exemplary frontal crash of the pick-up truck, shown with supporting structure. }
	\label{fig:truck_crash}  
\end{figure}

\subsection{Finite Element Signal Decomposition}
\label{decomp}
First, we investigate the quality and pro\-per\-ties of a representation of a function on a mesh stemming from crash simulations in the new basis. 
The approximation of the Laplace-Beltrami operator using geodesic distances is calculated as described in section \ref{sec:Lap_mesh} using the initial mesh configuration. Assuming the same connectivity as for the original mesh, we then compute the spectral coefficients for any mesh function $f$ with respect to the eigenvectors $\{\psi_j\}$ of the operator. The result is a set of coefficients in $\R^{N_h}$, where $N_h$ is the number of nodes of the mesh. 

\subsubsection{Decompositions of Deformation and Mesh Associated Variables}
\begin{figure}
\centering
\subcaptionbox{\label{fig:reconstructed_part}
 Using the first $p$ coefficients for reconstruction, $p=20$ (top), $p=100$ (bottom) and original data (middle)}{
\includegraphics[width=0.59\columnwidth]{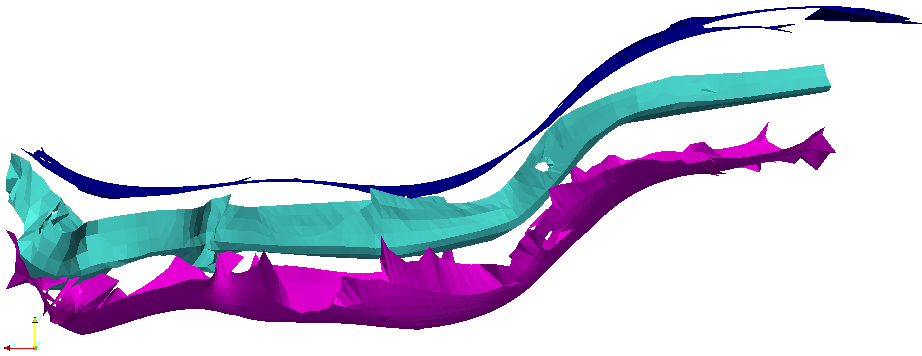}}\hfill
\subcaptionbox{\label{fig:reconstructed_coeff}
Magnitude of spectral coefficients for the $f^i_x$, $f^i_y$, and $f^i_z$ mesh functions.}{
\includegraphics[width=0.36\columnwidth]{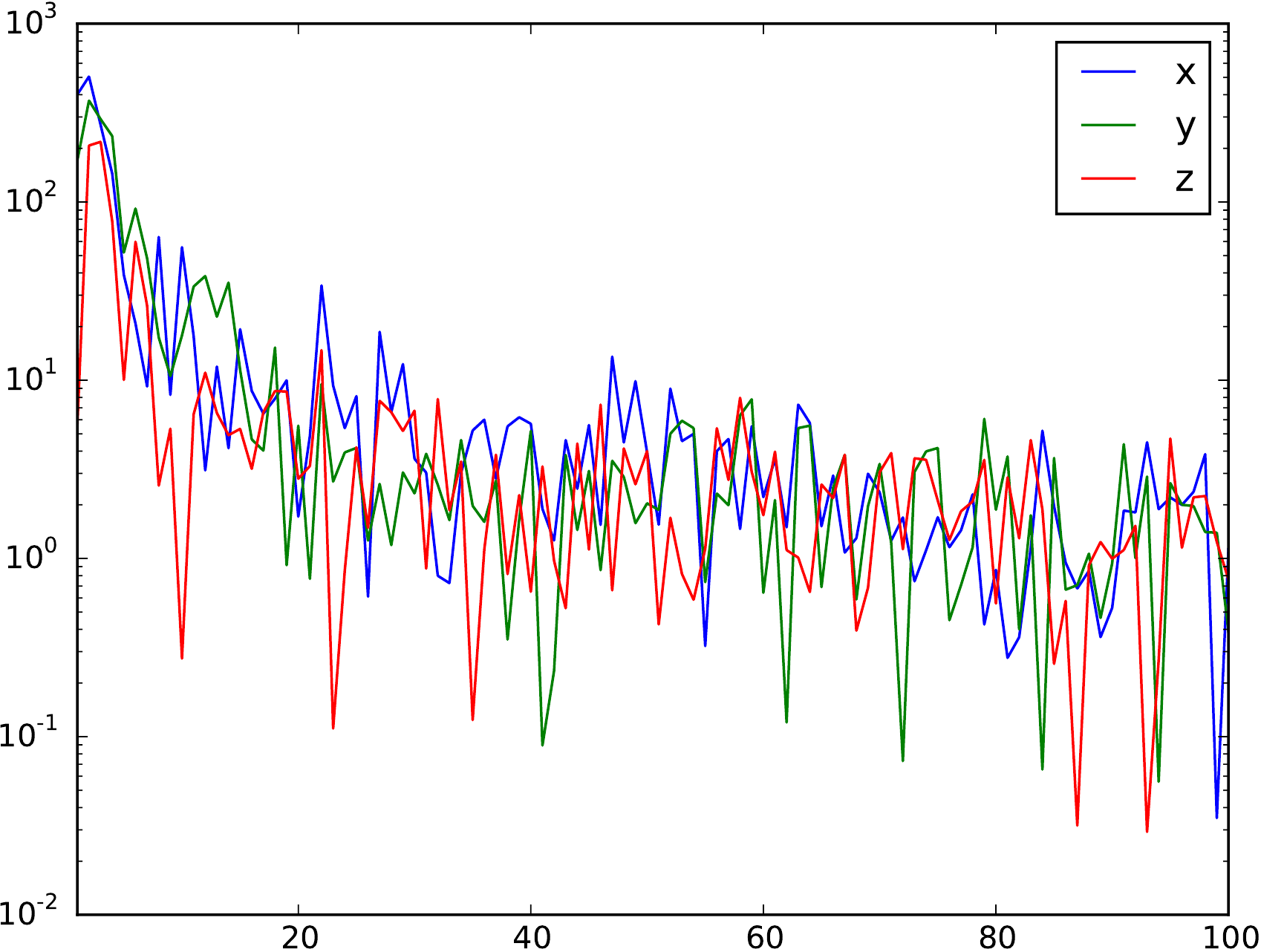}}
\caption{Analysis of a multi-scale reconstruction of a deformed shape.}
\label{Spectra}
\end{figure}

Given the mesh functions $f^i_x, f^i_y,$ and $f^i_z$, we can reconstruct the mesh deformations using  $f^i = \sum_{j=1}^{p} \alpha^i_j \psi_j$ for several values of $p$.
From figure \ref{fig:reconstructed_part}, it can be seen that using $p=20$ only a very coarse approximation of the part can be reconstructed. Adding more values, e.g. $p=100$, recovers more details of the part. 

It can also be clearly seen in figure \ref{fig:reconstructed_coeff} that most of the coefficients are small, with a few bigger ones. This is an essential feature which will be exploited for a classification task in section \ref{DM}.
We can use both the magnitude and the variance to identify a threshold for the significant components in the spectral decomposition.
Furthermore, we observed a similar empirical decay behavior for $f_x, f_y,$ and $f_z$ in other deformation applications. 
For example, very large deformations occur in the spectral coefficients arising for these mesh functions in a simulation of an airbag unfolding.

So far, we showed that the geometry of a car part can be represented with an orthogonal basis of the eigenvectors of an operator, by considering the $x,y,$ and $z$ coordinates of each mesh point as three separate mesh functions. 
But any function $f \in C^k(M)$ on the mesh can be represented by a linear combination of the eigenvectors. 
This implies that other variables associated to each node can also be represented using the same orthogonal basis. This can be used for variables like nodal strains, temperatures, velocities, and so on. 

We now demonstrate this with an example, where the nodal variable is the absolute difference of the movement of a mesh point between two time steps 
of the car crash simulation
\begin{equation}
\label{eq:fe_data_norm_diff}
f^i_k = \sqrt{\| u^i_k - v^i_k \|}, \text{ with } u^i_k,v^i_k \in R^{3}, k=1,\ldots,N_h=1714.
\end{equation}
Here $u^i_k,v^i_k$ denote the position of grid point $k$ of simulation $i$ at the time steps 6 and 7, respectively.

For figure \ref{Difference1}, the difference is reconstructed using different values of $p$. It can be seen that with $p=20$, compared with $p=200$ and with the original data, almost the original color distribution is obtained for the nodal variable. 
That is, for a qualitative analysis of the simulation results, the first 20 coefficients capture the essential part of the behavior.

\begin{figure}
\centering
\begin{subfigure}{0.3\columnwidth}
\includegraphics[width=1.0\columnwidth,height=1.0\columnwidth]{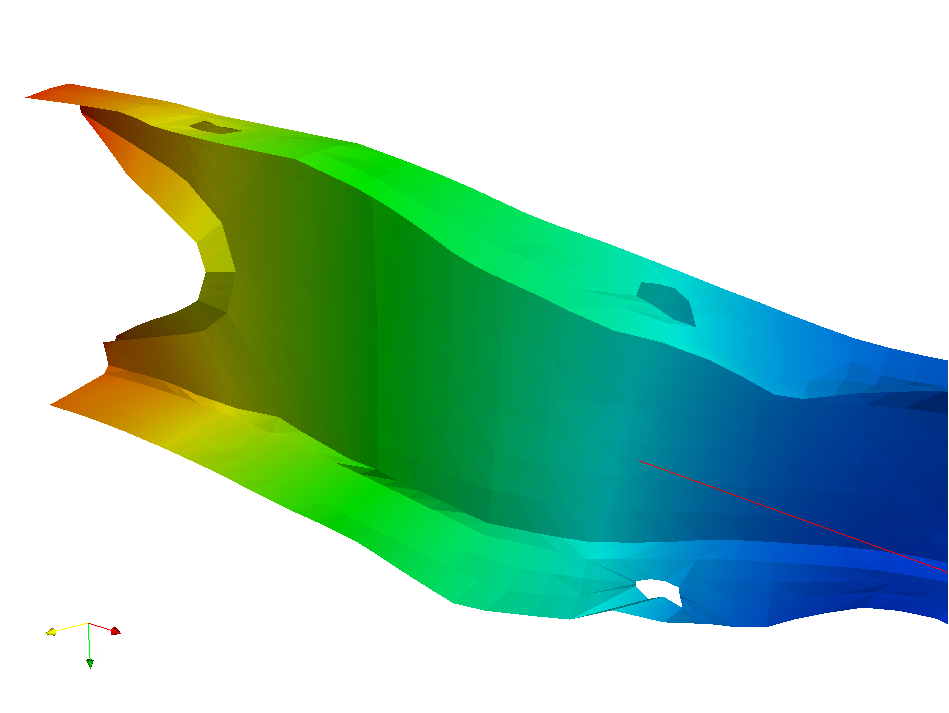}
\caption{$p=20$}
\end{subfigure}
\begin{subfigure}[c]{0.3\columnwidth}
\includegraphics[width=1.0\columnwidth,height=1.0\columnwidth]{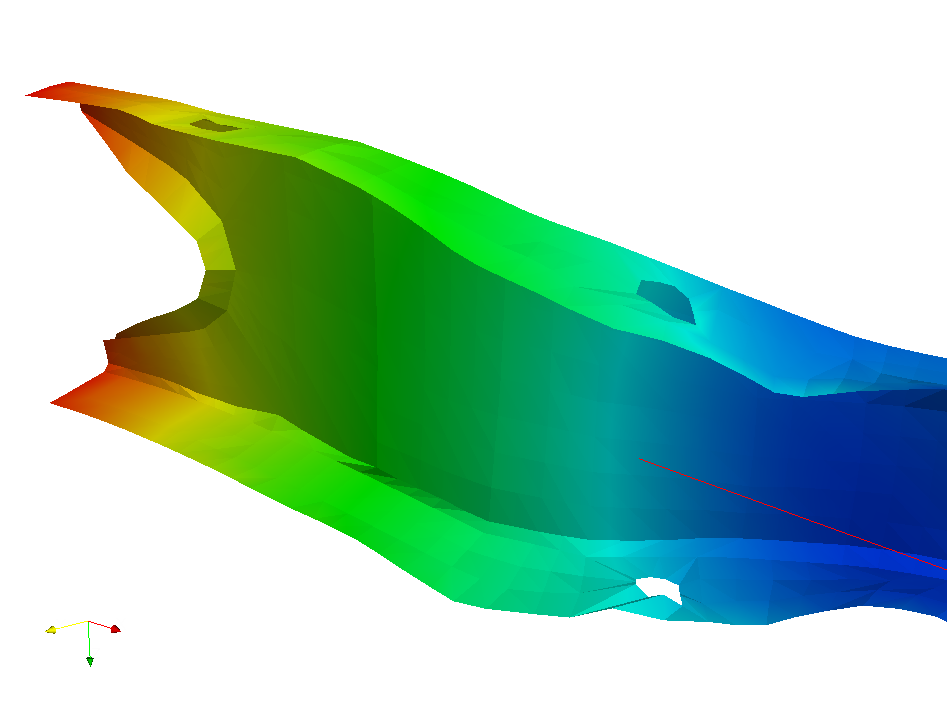}
\caption{$p=200$ }
\end{subfigure}
\begin{subfigure}{0.3\columnwidth}
\includegraphics[width=1.0\columnwidth,height=1.0\columnwidth]{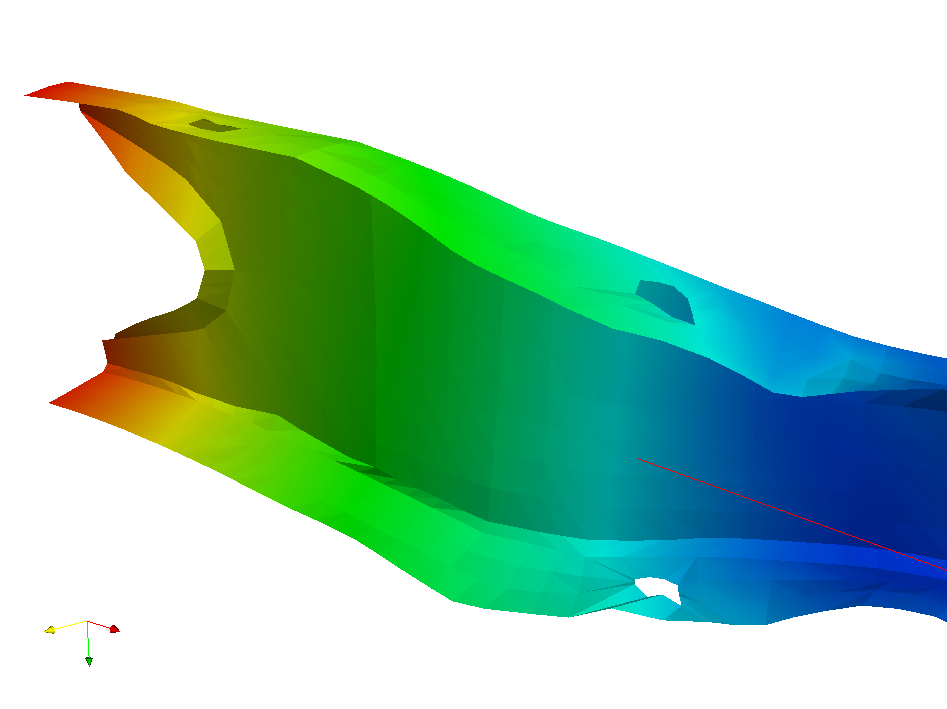}
\caption{original difference}
\end{subfigure}
\caption{Reconstruction of the differences between two time steps using $p$ coefficients. The color coding represents the absolute difference of the movement of a given mesh point. For a qualitative data analysis one captures with $p=20$ the essential behavior.}
\label{Difference1}
\end{figure}

\subsubsection{Data Analysis of Deformations in Crash Simulations}
\label{DM}
As in \cite{a2013}, we now investigate nonlinear dimensionality reduction on this data set.
Again, the norm of the difference of the deformations (\ref{eq:fe_data_norm_diff}) between the two time steps 6 and 7 was used as the nodal value of the mesh function. 
Based on this $m \times N_h$ dimensional data, a lower-dimensional embedding was computed using diffusion maps and other nonlinear dimensionality reduction procedures~\cite{a2013,DIzaTeran}.
With these data analysis approaches, we were able to successfully identify buckling modes and input parameter dependencies  

To demonstrate the usefulness of our approach, we use the spectral coefficients as input to the diffusion maps. 
As before, we now project all $m$ vectors $f^i$ along the eigenvectors of the obtained approximation of the Laplace-Beltrami operator. 
The coefficients decay very fast, similar to figure~\ref{fig:reconstructed_coeff}. 
As we presented in section 2, one can equivalently use these coefficients instead of the original mesh function of differences. 
We compare the embedding with diffusion maps using the first $p=4$ and $p=20$ coefficients with the original embedding obtained using all ${N_h}$ coefficients. 
Figure \ref{Embedding} shows a comparison of these embeddings, where each point corresponds to the embedding of one $f^i$. 
Note that one can observe a high correlation between one varied plate thickness and the clustering of the deformations.
\begin{figure}
\centering
  \begin{subfigure}{0.3\columnwidth}
    \includegraphics[width=1.0\columnwidth]{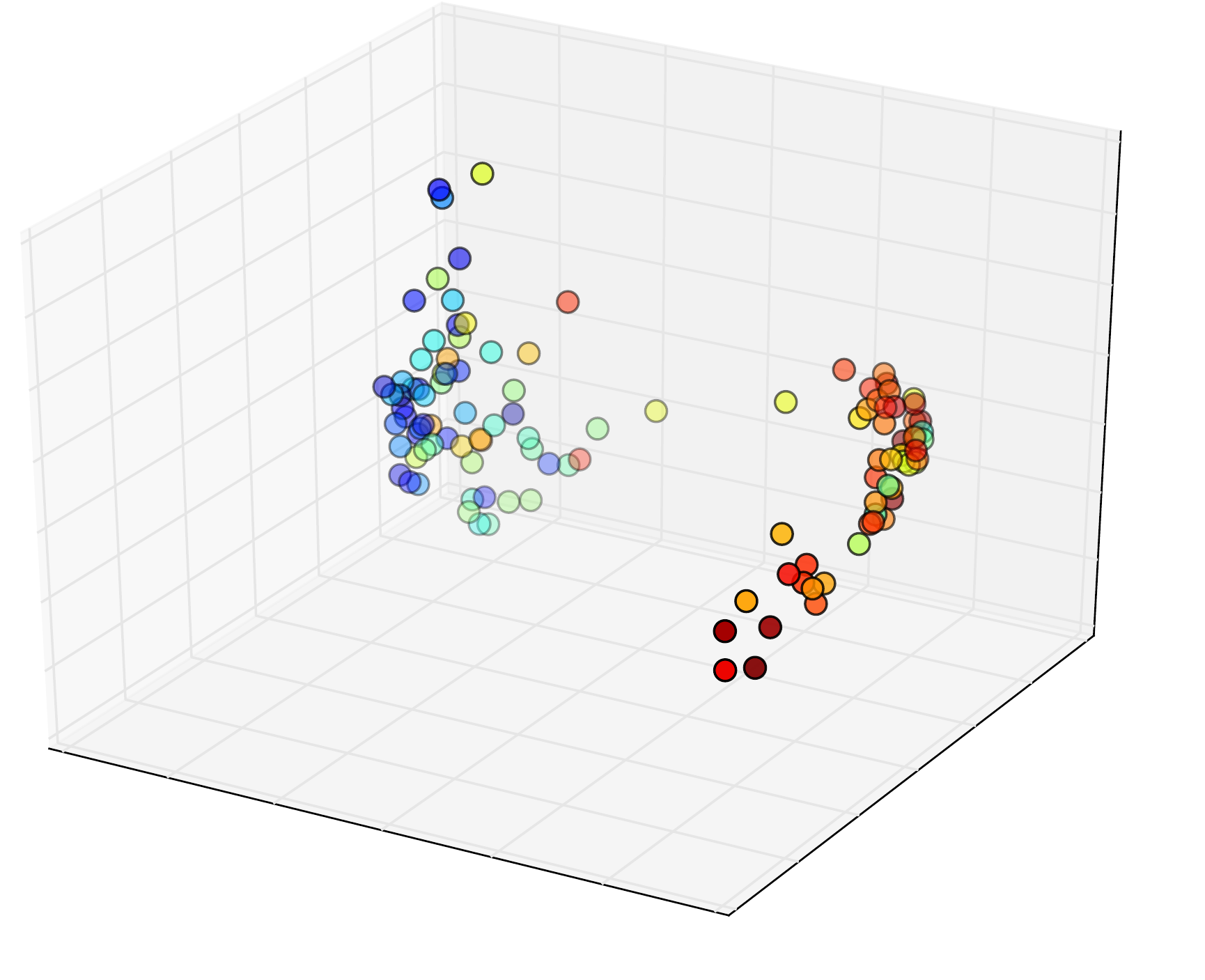}
    \caption{$p=4$}
  \end{subfigure}
  \begin{subfigure}[c]{0.3\columnwidth}
    \includegraphics[width=1.0\columnwidth]{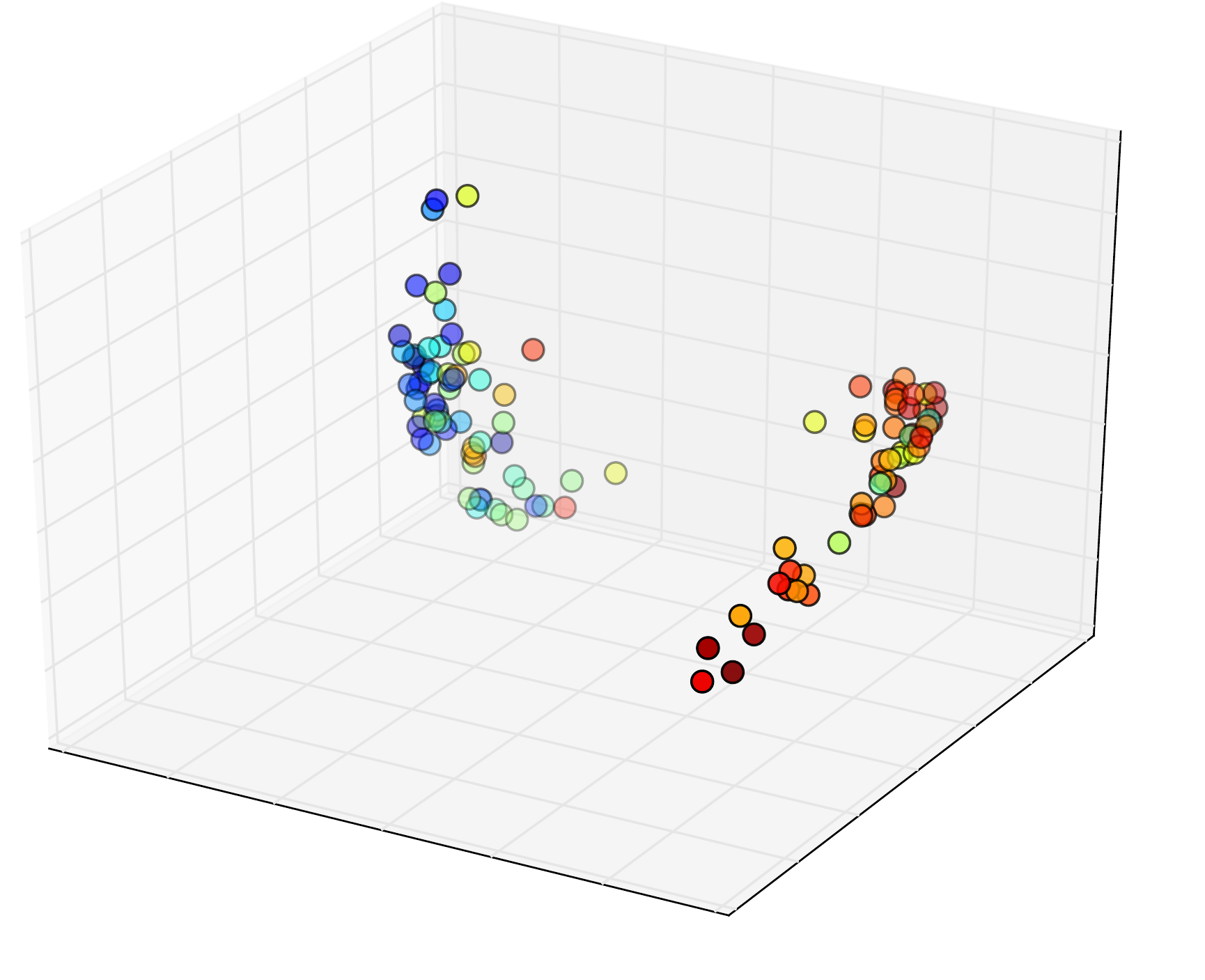}
    \caption{$p=20$}
  \end{subfigure}
  \begin{subfigure}{0.3\columnwidth}
    \includegraphics[width=1.0\columnwidth]{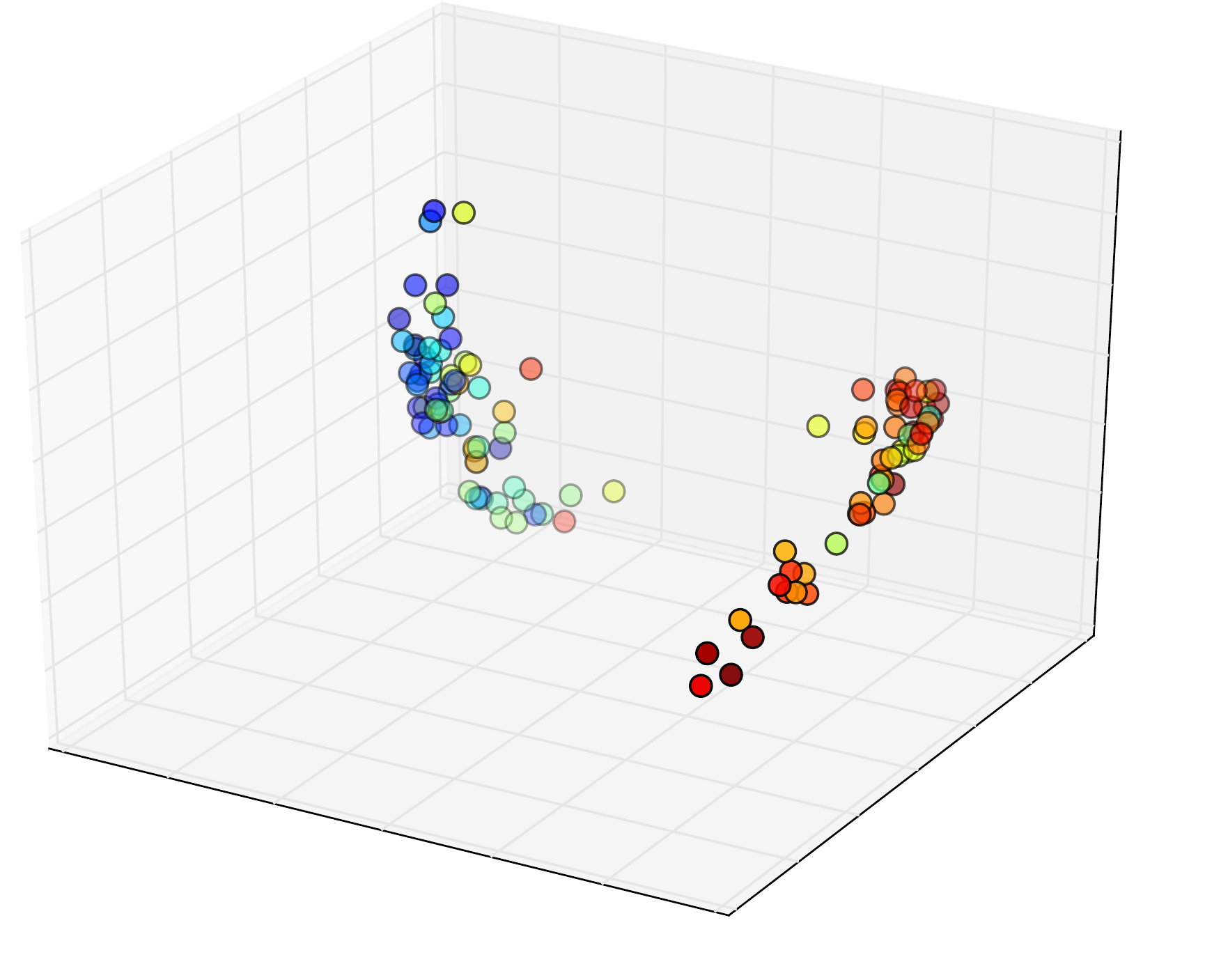}
    \caption{raw data}
  \end{subfigure}
	\caption{Comparison of diffusion maps embeddings for the car crash data using different number of projection coefficients with respect to the decomposition of the Laplace-Beltrami operator. 
	The color of the point for a simulation corresponds to the value of the plate thickness of a specific part.}
	\label{Embedding}
\end{figure}  
It is interesting to see that with only $p=4$, the structure of the embedding is almost the same as with using the original information of size $N_h=1714$. 
Therefore it can be used for classification of the bifurcation modes instead of the original one. 
This implies that the diffusion maps embedding is completely dominated by the first few components in the orthogonal decomposition, which correspond to coarse variations.

\subsection{Time Dependent Analysis of Crash Simulations}
\label{Bifurcation}
Due to the use of a common basis for all simulations and time steps, the introduced approach also allows an efficient analysis of time-dependent information.
Car crash simulations can vary strongly over time; the structure of a car can deform severely in very few milliseconds. Furthermore, an unstable behavior can originate from small variations in the material properties, initial load conditions, or numerically ill conditions. This phenomenon, buckling, is a serious problem for the robust design of car components. Relevant for an engineer is not only the identification of principal bifurcation modes, but also the study of the starting point of the unstable behavior. 
  
\begin{figure}
\begin{center}
\includegraphics[width=0.9\columnwidth]{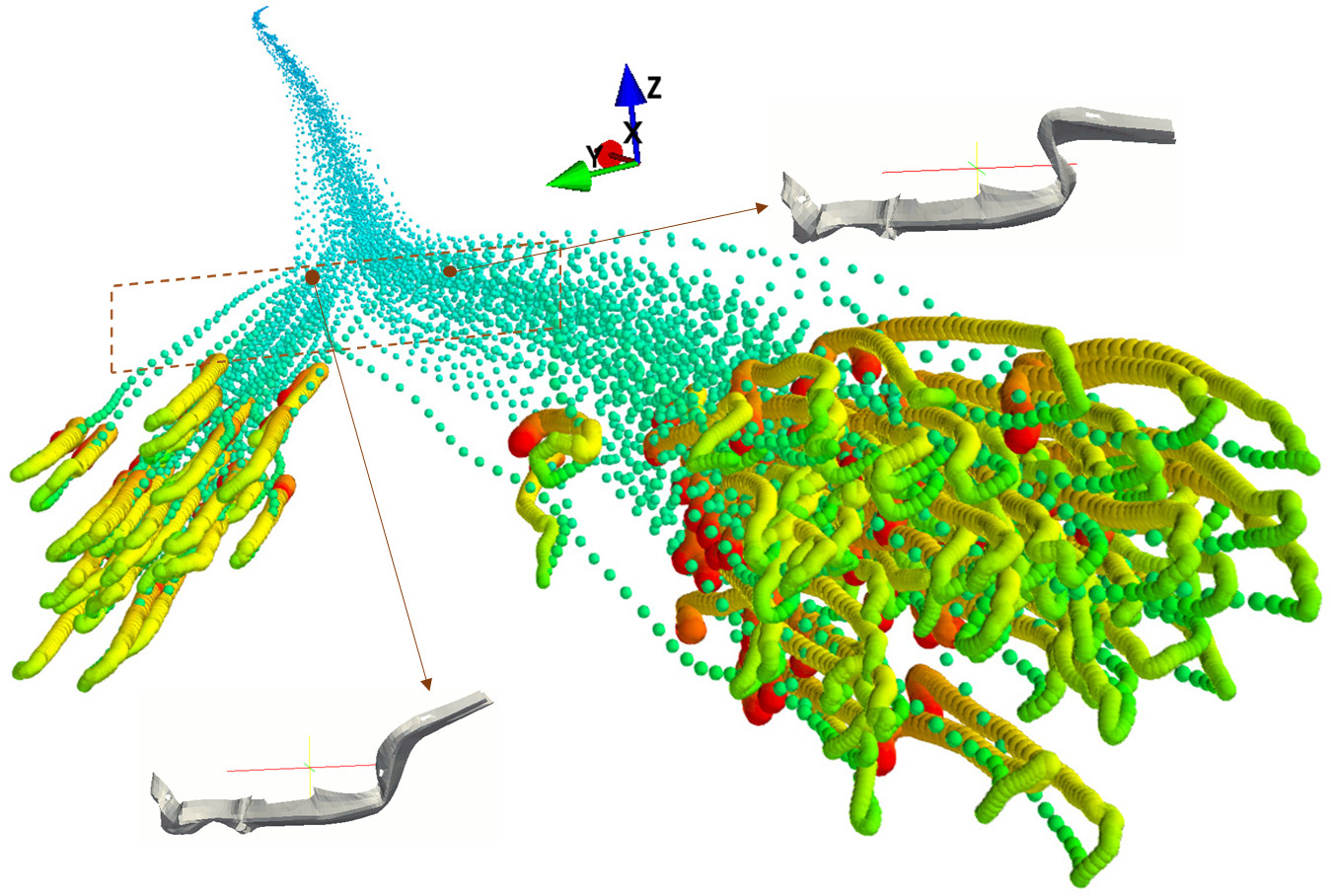}
\end{center}
\caption{Reduced 3D representation of $30400$ $(200 \text{ simulations} \times 152 \mbox{ time steps}) $ time-dependent simulation results, obtained by the spectral decomposition of the Laplace-Beltrami operator. 
Each point indicates a simulation at a specific time. 
The coordinates are the second spectral coefficients of the mesh functions $f_x^i$, $f_y^i$, and $f_z^i$ of each simulation at a time. 
The points are colored according to the corresponding time step of the simulation. Two deformation modes are clearly visible, as well as the approximate time of the bifurcation.}
\label{Time_path}  
\end{figure}

Our approach now allows a time-dependent analysis of such unstable deformation characteristics. 
We use the same Chevrolet truck example. 
To better visualize the time-dependent behavior, we employ more time steps and generate a second data set, consisting of $m=200$ simulations and $\tau=152$ time steps, where the same nine plate thicknesses as before are varied randomly.
Figure~\ref{Time_path} shows the deformation behavior over time for these simulations, where the color indicates the time. 
To obtain this embedding, we proceed as in section \ref{sec:method}, using the Laplace-Beltrami operator in algorithm \ref{alg:LaplaceBeltrami} for the initial mesh. 
Again, we use the obtained eigenvectors to project all simulations over all time steps by considering the mesh functions $f_x^i$, $f_y^i$, and $f_z^i$ in time. 
For the visual analysis, we choose the coefficients corresponding to the second component of the orthogonal decomposition as reduced coordinates.  

There are several observations that can be made from the obtained configuration in figure~\ref{Time_path}. We see that the reduced coordinates of the simulations give an organization of the data in time. 
The low dimensional representation shows two branches that correspond to two different deformation modes of a car part.
The bifurcation clearly starts about half-way through the crash simulation, while one can see how the positions of the simulations for the last time steps appear mixed with those before.
This corresponds with the rebound effect in a crash, where the car bounces back from the obstacle after the inertia of the movement was absorbed. 

Other embedding methods can be applied to the same data, but very few facilitate treating all time steps simultaneously.
The usual approach is to use PCA for time-dependent analysis as well. 
So as a first try, let us calculate a PCA using the data of all simulations, that is, $m$ vectors of size $3\cdot N_h$, from a time step where the bifurcation is clearly present. 
The spectral coefficients obtained by projecting the mesh functions $f_x^i$, $f_y^i$, and $f_z^i$ for some selected time steps along these principal components gets the low dimensional structure shown in figure \ref{fig:time_path_svd_one}, taking the first coefficient each. This approach does not produce adequate results since the variability over all time steps is not taken into account. 
Note that using other principal components does not change this observation.
Figure \ref{fig:time_path_svd_all} illustrates that computing a PCA using all simulations and time steps does improve the results, but only to some degree. 
A time behavior is now visible and a small separation is recognizable near the end. However, the clear separation and different results due to the bifurcation cannot be recognized in the embedding, in particular not the time of origin, as is the case in figure~\ref{Time_path}.  

\begin{figure}
\centering
\subcaptionbox{\label{fig:time_path_svd_one}
Embedding for a small number of time steps based on a PCA computed from one time step.}
{\includegraphics[width=0.32\columnwidth]{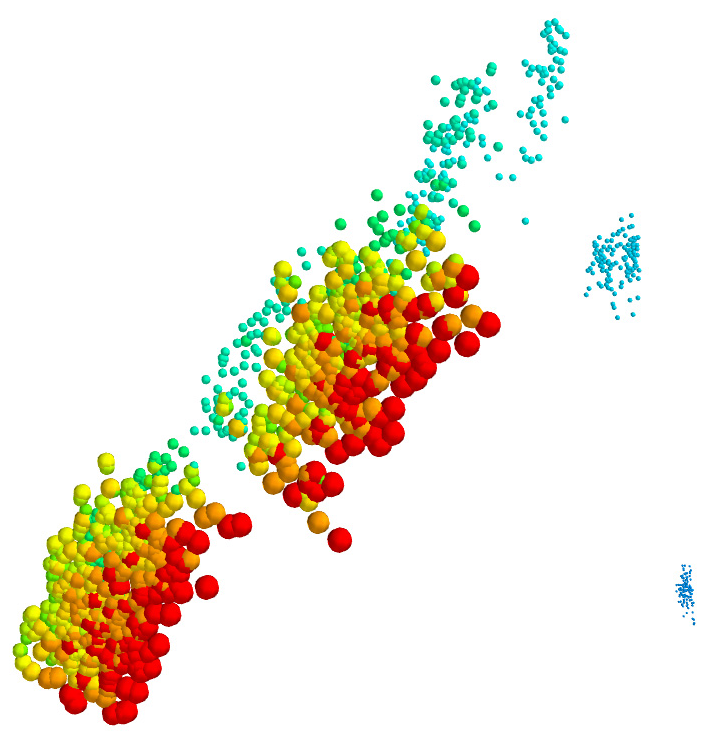}}
\hfill
\subcaptionbox{\label{fig:time_path_svd_all}
Embedding for all time steps based on a PCA computed from all time steps.}{
\includegraphics[height=0.33\columnwidth,width=0.32\columnwidth]{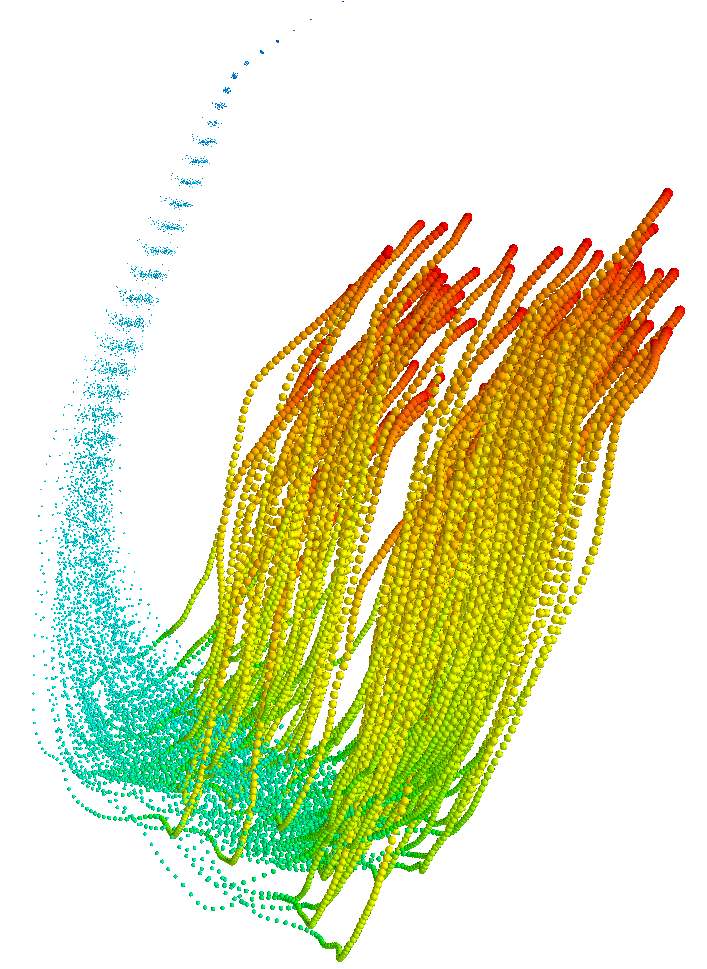}}
\hfill
\subcaptionbox{\label{fig:time_path_dm_all}
Diffusion maps embedding computed from all simulations at all time steps.}{
\includegraphics[height=0.33\columnwidth,width=0.32\columnwidth]{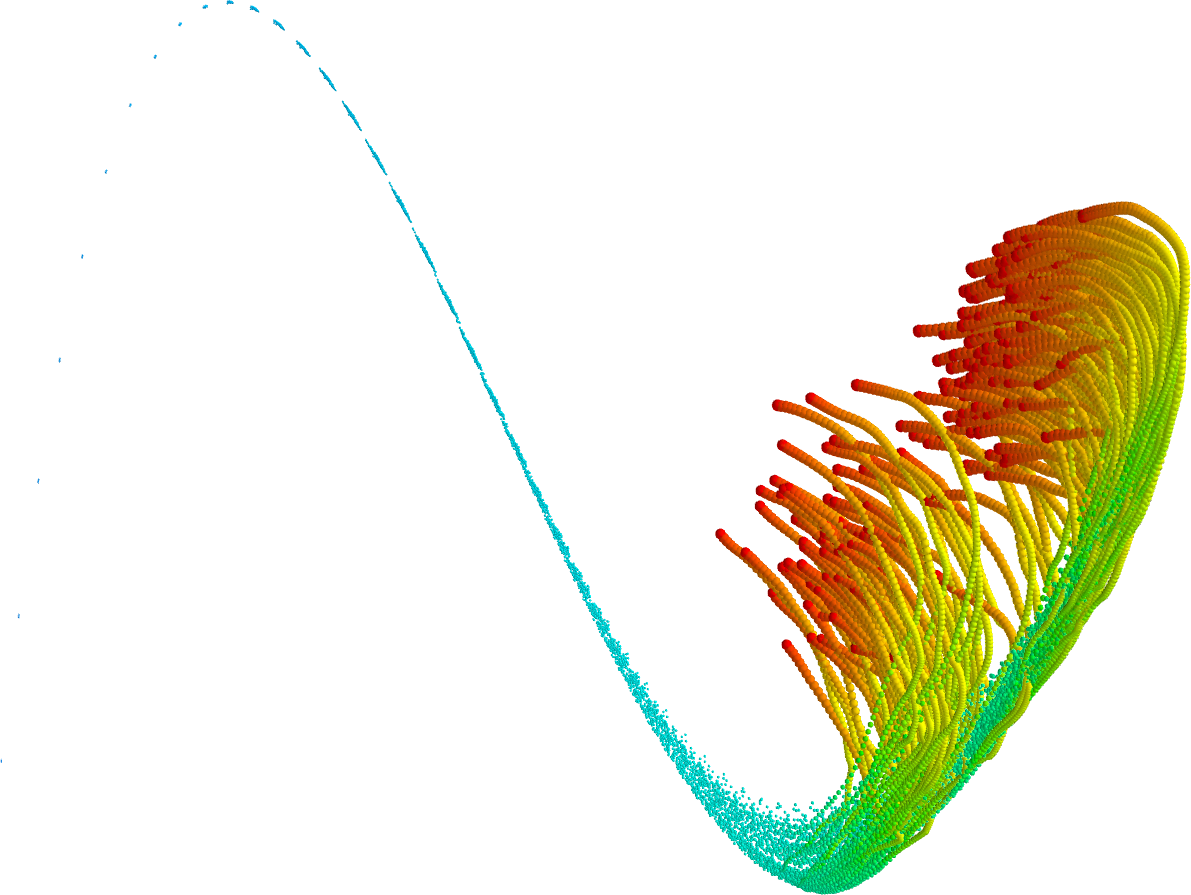}}
\caption{Reduced 3D representation of simulations obtained with PCA and diffusion maps. For PCA the coordinates are as before the first spectral coefficients of the mesh functions $f_x^i$, $f_y^i$, and $f_z^i$ of each simulation at a time. 
The points indicate a simulation at a specific time and are colored according to the corresponding time step.}
\label{Time_path_SVD} 
\end{figure}

In principle, other nonlinear dimensionality reduction methods could be used with this data set. 
One way would be to compute the embedding computed as many times as time steps are available.
But then the changes to and the switching of the eigenvectors make the recovering of a time-dependent low dimensional structure very cumbersome, if at all possible. 
Alternatively, one could attempt to embed other time steps into the coordinates obtained from one time step using the Nyström method.
However, this is as limited as the PCA example before. 
When using all simulations and all time steps for spectral nonlinear dimensionality reduction approaches, one has to deal with large full matrices of size $m\cdot\tau \times m\cdot\tau$, which can become infeasible when treating many simulation results. Furthermore, one would need to recompute the embedding when new data arrive that are not near the existing data, which is when using the Nyström method would become questionable otherwise.
As an example, we employ diffusion maps and compute the needed distance matrix by using, for each time step, a vector of length $3\cdot N_h$, consisting of the $x, y, z$-coordinates. 
The embedding obtained with diffusion maps is shown in figure~\ref{fig:time_path_dm_all}. It is somewhat similar to the PCA one.

In comparison, we need to compute an eigenvalue problem for a matrix of size $N_h \times N_h$ one time. 
This can be approximated by using hierarchical matrices. Afterwards we can project any new simulation result directly onto the obtained eigenbasis.
Note that besides the clearer interpretation of the embedding with our approach, the obtained coordinates are also more suitable as a distance measure. 
For example, consider the detection of an anomaly in the simulation results, where it is important that the distances between the two clusters are large.

\subsection{Crash Modes}
\begin{figure}
    \centering
    \begin{subfigure}{0.475\columnwidth}
        \includegraphics[width=0.9\columnwidth]{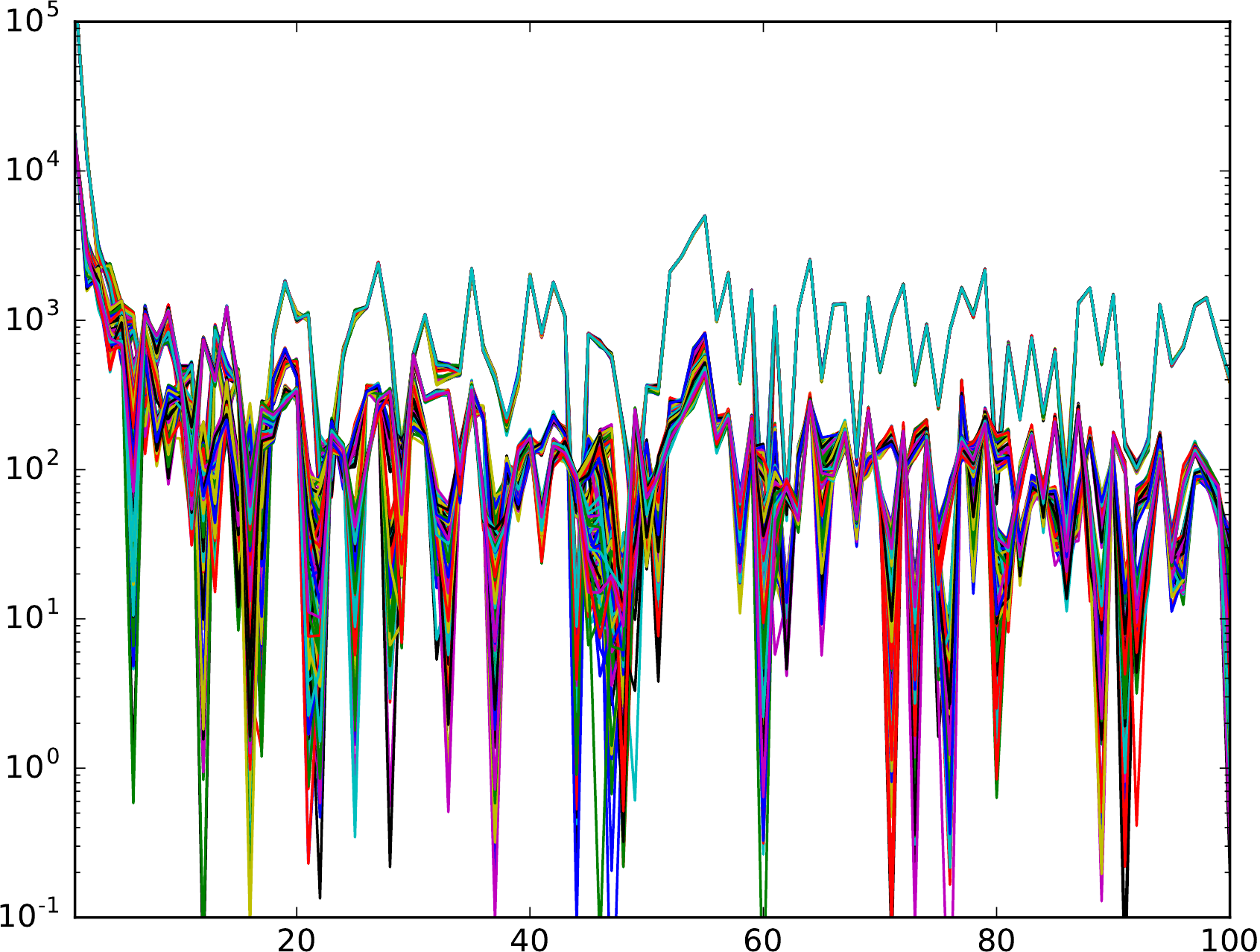}\\
        \caption{magnitude of coefficients}
        \label{fig:truck_NICA_spectral_coefficients}
    \end{subfigure}
    \begin{subfigure}[c]{0.475\columnwidth}
        \includegraphics[width=0.9\columnwidth]{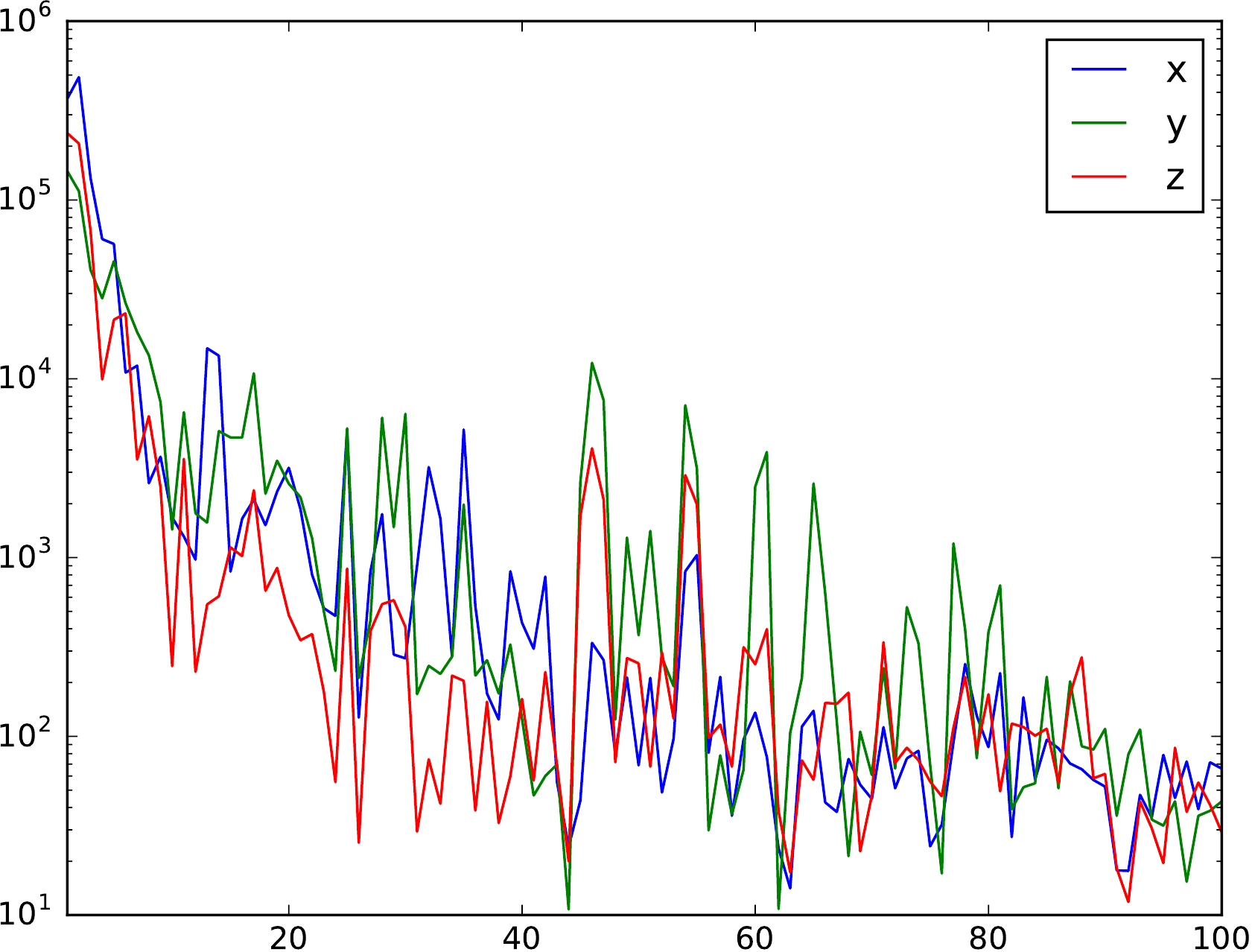}\\
        \caption{variance of coefficients}
        \label{fig:truck_NICA_variance_coefficients}
    \end{subfigure}
    \caption{Magnitude and variance of the first 100 spectral projection coefficients for the mesh functions $f^i_x$, $f^i_y$, and $f^i_z$ representing the deformation in a crash simulation, $i=1,\ldots,126$.}
    \label{fig:truck_NICA_coefficients}
\end{figure}

We again use the truck example with 17 time steps, but now study ways to identify and decompose different crash effects along so-called independent modes.
Again, we fix a specific time step (7 from 17) and compute from the shapes at this time step the discrete Fokker-Planck operator. At each mesh point, a cloud of $m$ points is formed from the $m$ displacements of this mesh point in the simulation bundle. That allows us to evaluate a local Jacobian and we can apply algorithm \ref{alg:algNICA}. 

In figure \ref{fig:truck_NICA_spectral_coefficients}, projection coefficients for mesh functions $f^i_x$, $f^i_y$, and $f^i_z$, $i=1,\ldots,m$, using the basis from the discrete Fokker-Planck operator from section \ref{sec:FP_operator} are shown. 
It can be observed that the magnitude is largest in the first, say, ten coefficients for all directions. 
Subsequently, it stays larger only for the $x$ component. 
This can be explained by the translation of the car in $x$ direction during the numerical crash simulation. Here the biggest changes are taking place. 
Nevertheless, the variance is reduced for all 3 mesh functions, as seen in figure \ref{fig:truck_NICA_variance_coefficients}. 
Note also, that both magnitude and variance are correlated between all three directions, since the basis is computed from the joint mesh.

In order to have a visual representation of the separation of effects by the spectral decomposition, we reconstruct the deformed geometry as a linear combination of the eigenvectors. 
A (re)construction of a ``virtual'' simulation reflecting only the transformations along a decomposition mode $p$ can be done by fixing the spectral coefficients for the directions $x$, $y$ and $z$ of a arbitrarily chosen simulation $j$, with exception of the $p$-th, for which the 3 coefficients are varied.
In other words, the spectral values for all but one component stem from one chosen simulation. 
The deformed geometry can be reconstructed as a linear combination of the eigenvectors. 
We consider the first four components of the orthogonal decomposition and show a few reconstructed simulations in figure~\ref{fig:truck_nica}.
As can clearly be seen, the first component reflects a translation. 
In the same fashion, we can observe that the second component corresponds to a rotation, the third to a global deformation, and the fourth to a local deformation.
The empirically observed behavior gives numerical evidence that group actions can be represented by projections along eigenvectors of invariant operators. 
Taking more coefficients, this can be extended to morph one simulation into another along the spectral components~\cite{DIzaTeran}. 

\begin{figure}
    \centering
    \subcaptionbox{translation\label{fig:truck_nica_translation_1}}{\includegraphics[width=0.475\columnwidth]{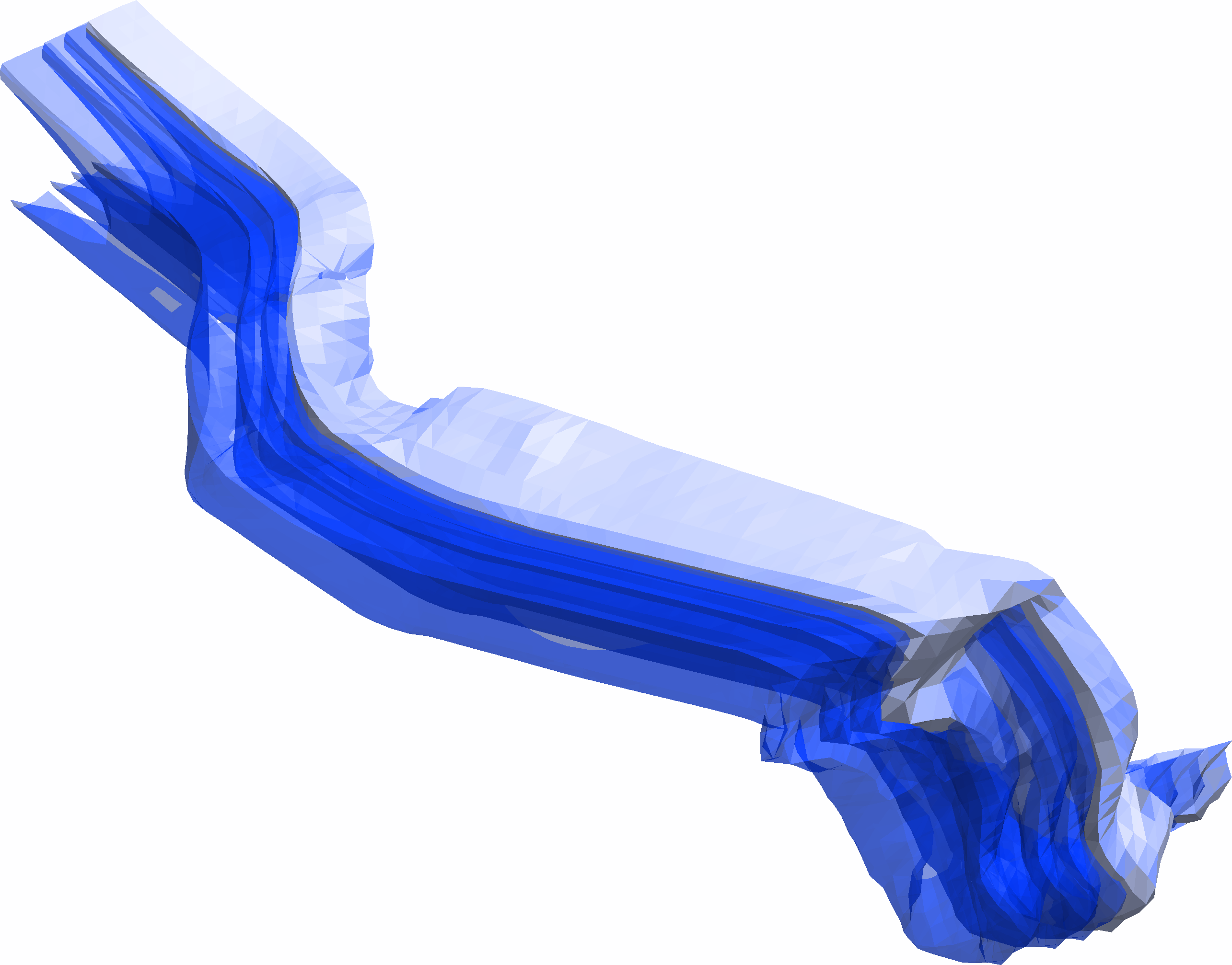}}
    \hfill
    \subcaptionbox{rotation\label{fig:truck_nica_rotation_2}}{\includegraphics[width=0.475\columnwidth]{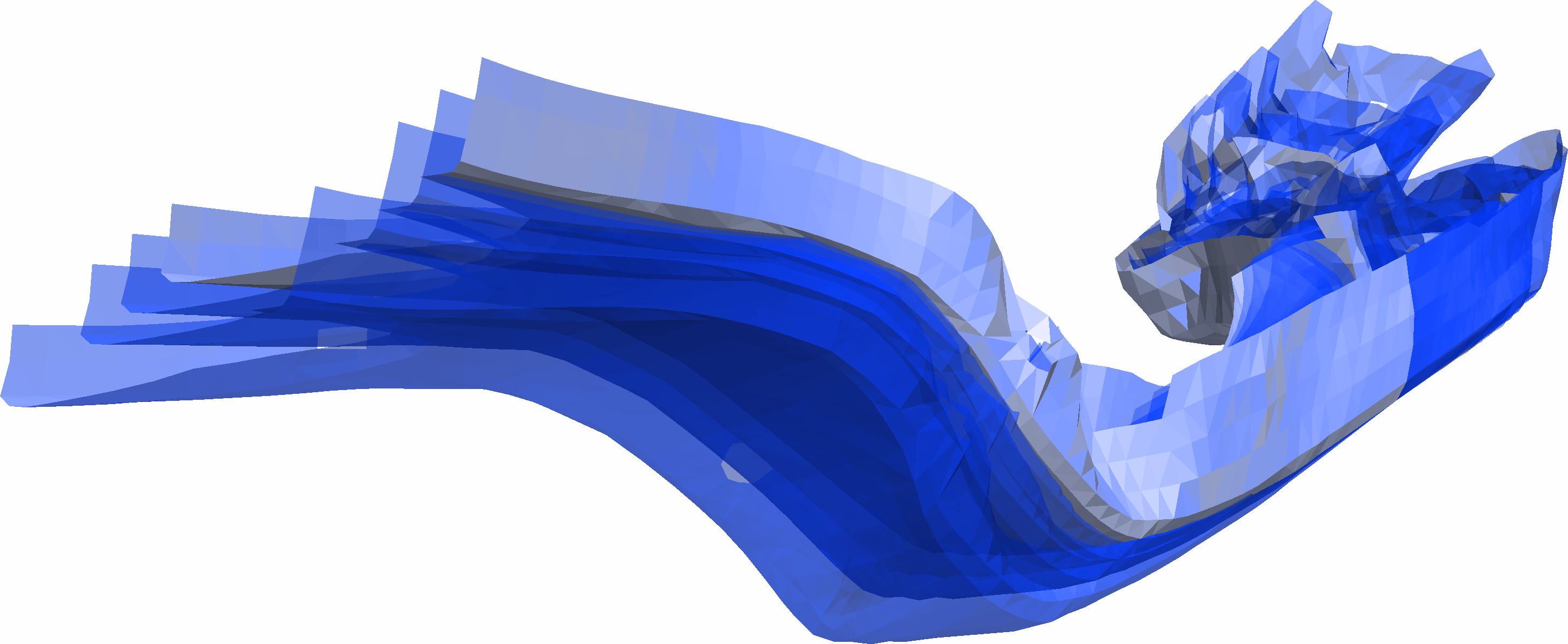}}\\
    \subcaptionbox{geometric deformation\label{fig:truck_nica_deformation_3}}{\includegraphics[width=0.475\columnwidth]{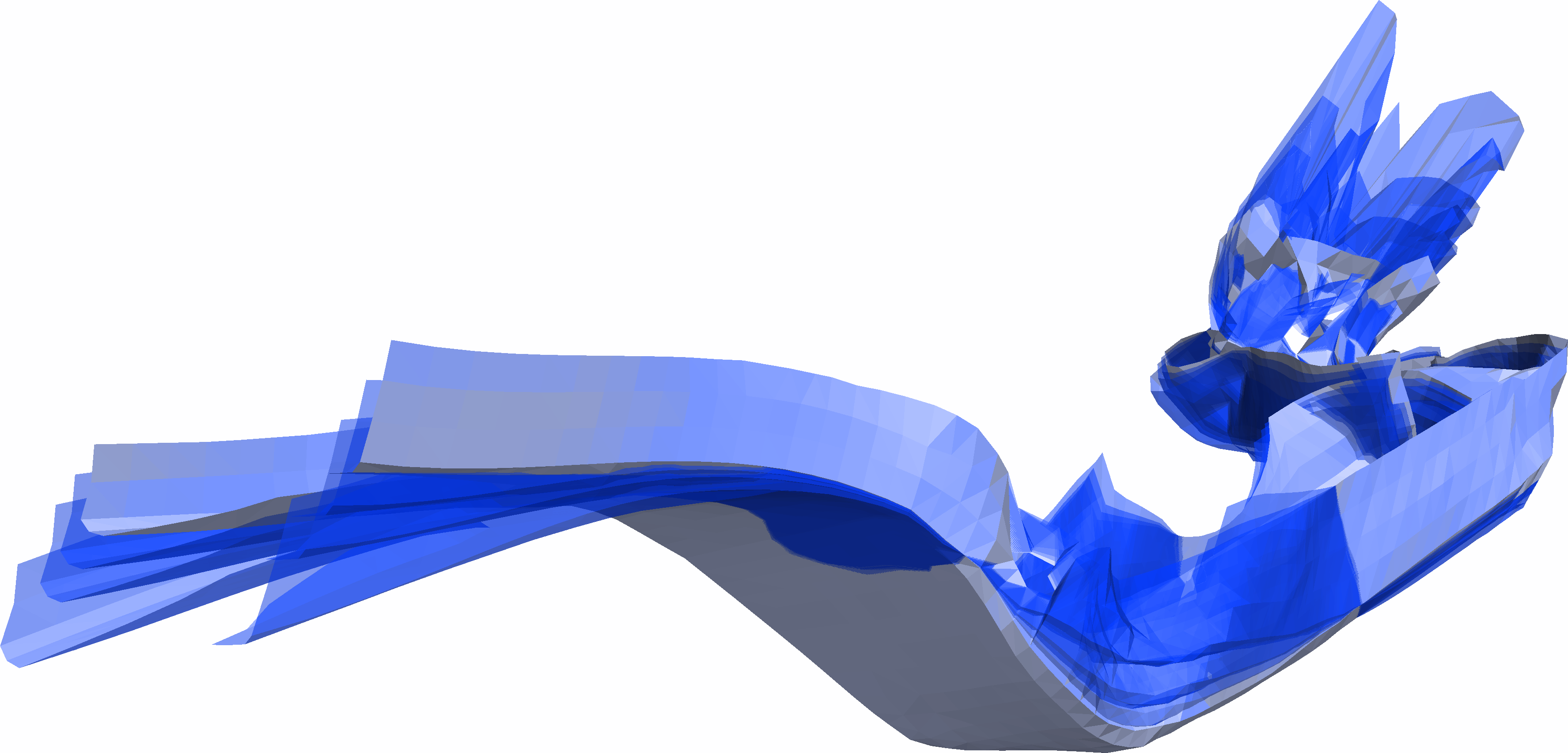}}
    \hfill
    \subcaptionbox{local geometric deformation\label{fig:truck_nica_deformation_4}}{\includegraphics[width=0.475\columnwidth]{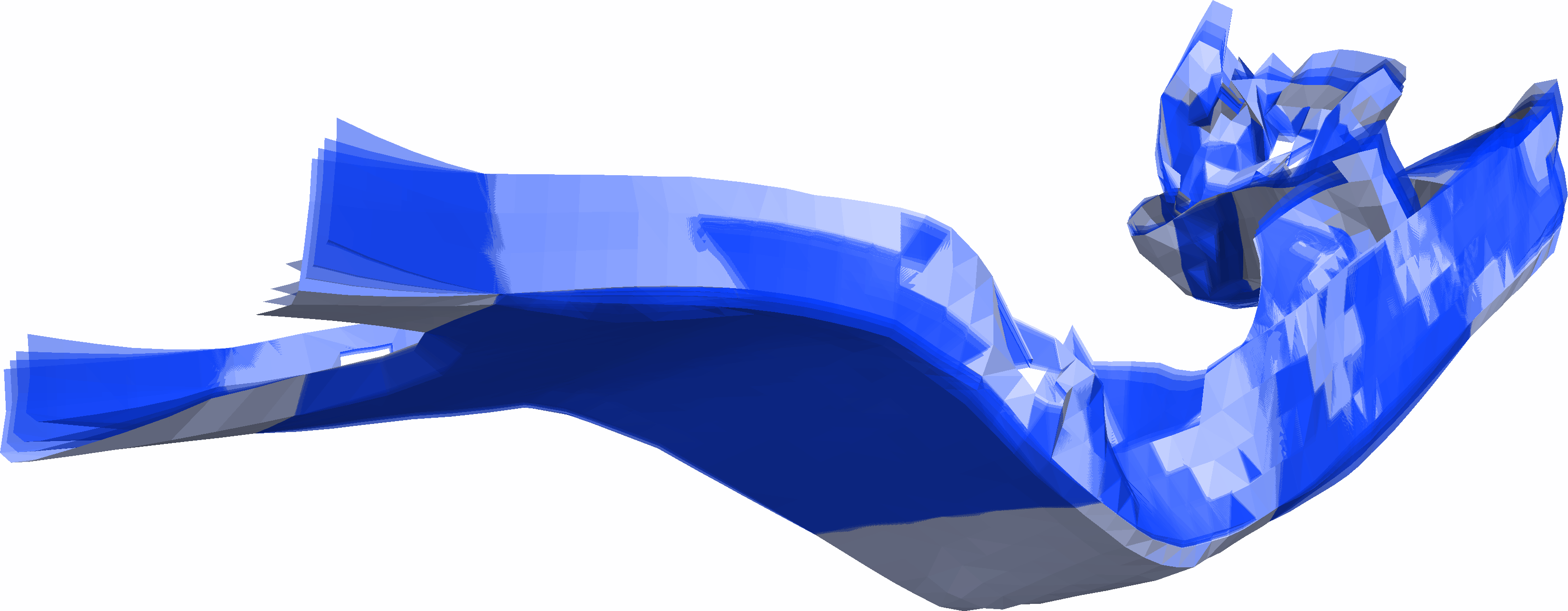}}
    \caption{Overlay of several reconstructions obtained when varying the coefficients with respect to the first four components 
        A translation corresponds to the first eigenvector, a rotation to the second, while the third and fourth eigenvectors reflect geometric deformations.}
\label{fig:truck_nica}
\end{figure}

\section{Summary and Discussion}
The investigation of bundles of numerical simulations, for example for industrial engineering applications, is a challenge due to the high dimensionality and complexity of the data. 
We introduced an analysis approach for this problem, which is based on a new ansatz for low dimensional data representation. 
The method is based on the assumption that all simulations are related by transformations. 
In particular, we propose an abstract setting where all simulations are in the quotient space of all embeddings of a manifold in $\R^3$ modulo a transformation group. 
The simulations are represented in a new basis, derived from an operator invariant to a specific transformation. 
Note that the underlying assumption of distance preserving transformations does not hold for a car crash where the studied parts break or the surfaces are elongated. Empirically, slight numerical disturbances of distances are unproblematic, but this needs to be studied in more detail.

We gave theoretical justifications that, under certain conditions, only a few components in this new basis are required to recover the essential behavior of the simulation data. This phenomenon is also observed for data stemming from engineering applications. 
Although we gave first ideas in this direction, several open problems remain.
In particular, a quantification of the decay of the spectral coefficients is needed.
Moreover, we lack an explanation of why using a discrete approximation of the Laplace-Beltrami operator or the Fokker-Planck operator allows such a decomposition into independent components, as observed in the numerical experiments.
Further research is warranted in order to analyze and to extend the assumptions of the theoretical setting or to analyze the use of operators invariant to other transformation. 
From another point of view, we have proposed an abstract setting that incorporates the general framework of shape analysis for this type of data. 
Here, a link to the theoretical treatment of known results for the quotient space $Emb(\Ms, \R^3)/\mathit{Diff}(\Ms)$ would be worthwhile. 
In particular further research has to be done, if one would like to efficiently evaluate geodesics in the simulation space, similar to the shape space setting.

The presented analysis of data from bundles of numerical simulations from car crash by the new method shows promising possibilities.
The use of the approach in other industrial contexts is to be explored. 
For example, we investigated the alignment of experimental measurement data in the form of point clouds from highly resolved 3D videos of a real car crash experiment in a crash facility with numerical simulation data. Here the aim is to find the corresponding numerical simulation to the measured point cloud data. For example, one aims to validate the numerical simulation approach or the involved material parameters. Such an alignment can be achieved with the eigenbasis obtained from a Fokker-Planck operator~\cite{Garcke.IzaTeran:2017}.
As an additional application consider the determination of optimal designs in a global optimization approach. Here, considerable efforts are invested in industrial product development, which is highly computationally intensive due to the large number of variables. We conjecture that the use of the new data representation has the potential to drastically reduce this complexity. 
In another context, the reduced basis method (RBM)~\cite{QuarteroniRBM} can be understood as a spectral method, where a problem dependent approximation basis is employed, which outlines a close relation to our approach.
Initial investigations for using the basis computed by our approach in an RBM-context do look promising~\cite{DIzaTeran}.

\section*{Acknowledgments}
We cordially thank the anonymous reviewers for their very helpful comments and suggestions.
\bibliographystyle{alphaabbr}
\bibliography{references}

\end{document}